\documentclass[12pt, a4paper]{amsart}

\usepackage{ amssymb, amsmath, enumerate, amsfonts, amsthm, mathrsfs, url, bm}

\numberwithin{equation}{section}

\usepackage{xcolor}  	
\usepackage[backref]{hyperref}
\hypersetup{
	colorlinks,
    linkcolor={blue!60!black},
    citecolor={blue!60!black},
    urlcolor={red!60!black}
}

\usepackage{color}

\usepackage[margin=1.25in]{geometry}

\usepackage[square,sort,comma,numbers]{natbib}
\newcommand{\C}{\mathbb{C}}

\newcommand{\N}{\mathbb{N}} 
\newcommand{\Z}{\mathbb{Z}}

\newcommand{\R}{\mathbb{R}}

\newcommand{\T}{\mathbb{T}}

\newcommand{\ve}{\varepsilon}

\newcommand{\f}[2]{\frac{#1}{#2}}

\newcommand{\set}[1]{\mathopen{}\left\{#1\mathclose{}\right\}}

\newcommand{\biggset}[1]{\biggl\{ #1 \biggr\}}

\newcommand{\abs}[1]{\mathopen{}\left| #1\mathclose{}\right|}
\newcommand{\bigabs}[1]{\bigl| #1 \bigr|}

\newcommand{\biggabs}[1]{\biggl| #1 \biggr|}

\newcommand{\sqbrac}[1]{\mathopen{}\left[ #1 \mathclose{}\right]}

\newcommand{\ceil}[1]{\mathopen{}\left\lceil #1 \mathclose{}\right\rceil}

\newcommand{\floor}[1]{\mathopen{}\left\lfloor #1 \right\rfloor}
\newcommand{\brac}[1]{\mathopen{}\left( #1 \mathclose{}\right)}
\newcommand{\bigbrac}[1]{\bigl( #1 \bigr)}
\newcommand{\Bigbrac}[1]{\Bigl( #1 \Bigr)}
\newcommand{\biggbrac}[1]{\biggl( #1 \biggr)}

\newcommand{\norm}[1]{\mathopen{}\left\| #1\mathclose{}\right\|}
\newcommand{\bignorm}[1]{\big\| #1 \big\|}

\newcommand{\ang}[1]{\mathopen{}\left\langle#1\mathclose{}\right\rangle}

\newcommand{\recip}[1]{\frac{1}{#1}}

\newcommand{\vh}{\underline{h}}

\newcommand{\E}{\mathbb{E}}

\newcommand{\intd}{\mathrm{d}}
\newcommand{\supp}{\mathrm{supp}}

\newcommand{\eps}{\varepsilon}
\newcommand{\hash}{\#}

\makeatletter
\let\@@pmod\pmod
\DeclareRobustCommand{\pmod}{\@ifstar\@pmods\@@pmod}
\def\@pmods#1{\mkern4mu({\operator@font mod}\mkern 6mu#1)}
\makeatother

\newtheorem{theorem}{Theorem}[section]

\newtheorem{corollary}[theorem]{Corollary}
\newtheorem{claim}[theorem]{Claim}

\newtheorem{lemma}[theorem]{Lemma}

\theoremstyle{definition}

\newtheorem*{remark}{Remark}

\numberwithin{theorem}{section}

\title{Quantitative bounds in the nonlinear Roth theorem}

\author{Sarah Peluse}
\thanks{During the completion of this work, the first author was partially supported by the NSF Graduate Research Fellowship Program under Grant No.\ DGE-114747 and by the Stanford University Mayfield Graduate Fellowship.}
\address{School of Mathematics, Institute for Advanced Study, Princeton, 
USA}
\address{and}
\address{Department of Mathematics, Princeton University, Princeton, 
USA}
\email{speluse@princeton.edu}

\author{Sean Prendiville}
\address{Department of Mathematics and Statistics, Lancaster University, UK}
\email{s.prendiville@lancaster.ac.uk}

\begin{document}

\begin{abstract}
We show that there exists $c>0$ such that any subset of $\{1, \dots, N\}$ of density at least $(\log\log{N})^{-c}$ contains a nontrivial progression of the form $x,x+y,x+y^2$. This is the first quantitatively effective version of the Bergelson--Leibman polynomial Szemer\'edi theorem for a progression involving polynomials of differing degrees. Our key innovation is an inverse theorem characterising sets for which the number of configurations $x,x+y,x+y^2$ deviates substantially from the expected value. In proving this, we develop the first effective instance of a concatenation theorem of Tao and Ziegler, with polynomial bounds.
\end{abstract}

\maketitle

\setcounter{tocdepth}{1}
\tableofcontents

\maketitle

\section{Introduction}\label{introduction}

Gowers~\cite[Problem 11.4]{GowersArithmetic} has posed the problem of obtaining quantitative bounds in the polynomial Szemer\'edi theorem of Bergelson and Leibman~\cite{BergelsonLeibmanPolynomial}. This states that if $P_1,\dots,P_m\in\Z[y]$ all have zero constant term, then any subset of $\set{1,2, \dots, N}$ lacking the polynomial progression
\begin{equation}\label{eq1.1}
x,\ x+P_1(y),\ \dots\ ,\ x+P_m(y) \qquad (y \in \Z\setminus\set{0})
\end{equation}
has size $o(N)$. Hitherto, all effective versions of this result have been restricted to two-term progressions \cite{SarkozyDifferenceI,SarkozyDifferenceIII, PSSSets, BPPSDifference, SlijepcevicPolynomial, LucierIntersective, RiceMaximal}, arithmetic progressions with common difference equal to a perfect power \cite{GowersNewFour,GowersNew,PrendivilleQuantitative}, or are concerned with the analogous question over finite fields \cite{BourgainChangNonlinear, PeluseThree, DLSImproved, PelusePolynomial}. In this paper, we obtain the first  bound over the integers for a progression of length greater than two and involving polynomials of differing degrees.

\begin{theorem}\label{main}
There exists $c>0$ such that if $A \subset \set{1,2 ,\dots, N}$ contains no progression of the form\footnote{We call this the \emph{nonlinear Roth} configuration, after Bourgain and Chang \cite{BourgainChangNonlinear}.}
\begin{equation}\label{main config}
x,\ x+y, \ x+y^2 \qquad (y \neq 0),
\end{equation}
then\footnote{See \S\ref{sec:notation} for our conventions regarding asymptotic notation such as `$\ll$'.}
\begin{equation}\label{eq:density}
|A|\ll N(\log\log{N})^{-c}.
\end{equation}
\end{theorem}
\begin{remark}\label{rmk1.2}
Keeping track of exponents in our proof, $c = 1/2^{150}$ is admissible.
\end{remark}

%


Our proof of Theorem~\ref{main} adapts a strategy of the first author \cite{PelusePolynomial} from finite fields to the integer setting. There are multiple issues with applying these ideas  in the integers, so the proof of Theorem~\ref{main} requires several significant modifications and additions. The key insight of~\cite{PelusePolynomial} is that if one can control the count of an affine independent polynomial progression  by the Gowers $U^s$-norm, then one can use this and an understanding of shorter progressions to prove control of the count by the $U^{s-1}$-norm. Thus, if one understands shorter progressions and can show control by any $U^s$-norm, then one can deduce control by the $U^1$-seminorm. As the $U^1$-seminorm measures correlation with constant functions, this is very powerful information. 

Over the integers there are certain `local' issues which preclude effective control by the global $U^1$-seminorm. Instead, we control  our counting operator by an average of $U^1$-seminorms, each localised to a progression of length $ N^{1/2}$ and small common difference.

\begin{theorem}[Inverse theorem for nonlinear Roth]\label{thm:inverse-zero}
Let $f : \Z \to \C$ be a 1-bounded function supported  in the interval $[N]:=\set{1,\dots, N}$ and let $\delta >0$. Suppose that
$$
\abs{\sum_{x\in \Z}\, \sum_{y \in [N^{1/2}]} f(x)f(x+y)f(x+y^2)} \geq \delta N^{\frac{3}{2}}.
$$
Then either $N\ll \delta^{-O(1)}$ or there exist positive integers $q\ll\delta^{-O(1)}$ and $ N'\gg\delta^{O(1)}N^{1/2}$ such that
\begin{equation}\label{0eq3.4}
\sum_{x\in \Z}\left|\sum_{y\in[N']}f(x+qy)\right|\gg\delta^{O(1)}NN'.
\end{equation}
\end{theorem}

To derive our density bound (Theorem~\ref{main}), we use our inverse theorem (Theorem \ref{thm:inverse-zero}) to prove that sets lacking the nonlinear Roth configuration have a density increment on a progression with  small common difference. In a sequel \cite{PelusePrendivillePolylogarithmic} we further bootstrap our inverse theorem (Theorem \ref{thm:inverse-zero}) to obtain a  larger density increment, and thereby replace the double logarithm in \eqref{eq:density} with a single logarithm.  This bootstrapping procedure follows the (now standard) energy increment procedure of Heath--Brown and Szemer\'edi \cite{HeathBrownInteger, SzemerediInteger}, and to avoid obfuscating our argument with further technicalities, we delegate this improvement to a subsequent paper.

A variety of perspectives, both ergodic and combinatorial, can be used to establish that the count of nonlinear Roth configurations is controlled by a $U^1$-seminorm, in a qualitative sense. The novelty of Theorem \ref{thm:inverse-zero} is that this is demonstrated in a quantitatively effective manner, with polynomial bounds, by avoiding standard tools of higher order Fourier analysis (which give poorer bounds). Indeed, whilst Gowers norms of high degree such as the $U^5$-norm play a role in our argument, we completely avoid using the inverse theorem for these norms, the equidistribution theory of nilsequences, or any version of the arithmetic regularity lemma, requiring only Fourier analysis and numerous applications of the Cauchy--Schwarz inequality.

Perhaps the biggest difficulty in adapting the argument of~\cite{PelusePolynomial} to the integer setting is in first showing that the count of nonlinear Roth configurations is controlled by some global $U^s$-norm. While this is not too difficult to accomplish in finite fields using Bergelson and Leibman's PET induction scheme~\cite{BergelsonLeibmanPolynomial}, in the integers such an argument yields control in terms of an average of certain constrained Gowers norms. We must then show, with quantitative bounds, that this average of constrained Gowers norms is controlled by a genuine global $U^s$-norm. Such `concatenation' results have been proved by Tao and Ziegler~\cite{TaoZieglerConcatenation}; however the quantitative dependence in their argument is (at best) tower-type.   We offer a different proof of an instance of their concatenation theorem, one which yields  polynomial bounds.

This paper is organised as follows. In \S\ref{sec3}, we give a more detailed outline of the proof of Theorem~\ref{main}. In \S\ref{increment sec} we derive Theorem \ref{main} from a density increment lemma, whose proof is deduced from a generalisation of our inverse theorem (Theorem \ref{thm:inverse-zero}) in \S\ref{increment proof sec}. This generalised inverse theorem (Theorem \ref{thm:inverse-one}) is proved in \S\ref{inverse theorem section}.  In \S\ref{pet} we show that our counting operator is controlled by an average of constrained Gowers norms, and in \S\ref{concat sec} we show that these constrained averages are controlled by a single global Gowers norm of higher degree. This `concatenation' argument uses an arithmetic variant of the box norm inverse theorem, which we state and prove in \S\ref{arithmetic inverse}. Finally in \S\ref{sec6} we describe our degree lowering procedure, showing how global control of our configuration by the $U^s$-norm, implies global control by the $U^{s-1}$-norm.

\subsection*{Acknowledgements}
We thank Mariusz Mirek for numerous corrections.

\subsection{Notation}\label{sec:notation}
\subsubsection{Standard conventions}
We use $\N$ to denote the positive integers.  For real $X \geq 1$, write $[X] = \{ 1,2, \ldots, \floor{X}\}$.  A complex-valued function is \emph{1-bounded} if the modulus of the function does not exceed 1.

We use counting measure on $\Z$, so that for $f,g :\Z \to \C$ we have
$$
\ang{f,g} := \sum_x f(x)\overline{g(x)}\qquad \text{and}\qquad \norm{f}_{L^p} := \biggbrac{\sum_x |f(x)|^p}^{\recip{p}}.
$$ 
Any sum of the form $\sum_x$ is to be interpreted as a sum over $\Z$. 
We use Haar probability measure on $\T := \R/\Z$, so that for measurable $F : \T \to \C$ we have
$$
\norm{F}_{L^p} := \biggbrac{\int_\T |F(\alpha)|^p\intd\alpha}^{\recip{p}} = \biggbrac{\int_0^1 |F(\alpha)|^p\intd\alpha}^{\recip{p}}
$$
For $\alpha \in \T$ we write $\norm{\alpha}$ for the distance to the nearest integer.  

For a finite set $S$ and function $f:S\to\C$, denote the average of $f$ over $S$ by
\[
\E_{s\in S}f(s):=\f{1}{|S|}\sum_{s\in S}f(s).
\]

Given functions $f,g : G \to \C$ on an additive group with measure $\mu_G$ we define their convolution by 
\begin{equation}\label{convolution}
f*g(x) := \int_G f(x-y) g(y) \intd\mu_G,
\end{equation}
when this makes sense.

We define the Fourier transform of $f : \Z \to \C$ by 
\begin{equation}\label{Fourier transform}
\hat{f}(\alpha) := \sum_x f(x) e(\alpha x) \qquad (\alpha \in \T),
\end{equation}
again, when this makes sense.  Here $e(\alpha)$ stands for $e^{2\pi i \alpha}$.

The \emph{difference function} of $f : \Z \to \C$ is the function $\Delta_h f : \Z \to \C$ given by 
\begin{equation}\label{eq:diff}
\Delta_hf(x) = f(x) \overline{f(x+h)}.
\end{equation}
Iterating, we set
$$
\Delta_{h_1, \dots, h_s} f := \Delta_{h_1} \dots \Delta_{h_s} f.
$$
This allows us to define the \emph{Gowers $U^s$-norm}
\begin{equation}\label{Us def}
\norm{f}_{U^s} := \brac{\sum_{x, h_1, \dots, h_s} \Delta_{h_1, \dots, h_s} f(x)}^{1/2^s}.
\end{equation}
When $S \subset \Z$ we define the \emph{localised} Gowers $U^s$-norm
\begin{equation}\label{local Us def}
\norm{f}_{U^s(S)} := \norm{f1_S}_{U^s}.
\end{equation}
Notice that the left-hand side of \eqref{0eq3.4} is equal to 
$$
\sum_x \norm{f}_{U^1(x + q\cdot [N'])}.
$$


For a function $f$ and positive-valued function $g$, write $f \ll g$ or $f = O(g)$ if there exists a constant $C$ such that $|f(x)| \le C g(x)$ for all $x$. We write $f = \Omega(g)$ if $f \gg g$.  We sometimes opt for a more explicit approach, using $C$ to denote a large absolute constant, and $c$ to denote a small positive absolute constant.  The values of $C$ and $c$ may change from line to line. 
\subsubsection{Local conventions}
Up to normalisation, all of the above are well-used in the literature. Next we list notation specific to our paper. We have tried to minimise this in order to aid the casual reader.  

The quantity $(N/q)^{1/2}$ appears repeatedly in our arguments, where $q$ is an integer fixed throughout the majority of our paper. Unless otherwise specified, we therefore adopt the convention that
\begin{equation}\label{M def}
M:= \floor{\sqrt{N/q}}.
\end{equation}

Define the \emph{counting operator} on the functions $f_i : \Z \to \C$  by
\begin{equation}\label{counting op}
\Lambda_q(f_0, f_1, f_2) := \sum_{x \in \Z} \sum_{y \in \N} f_0(x)f_1(x+y) f_2(x+qy^2).
\end{equation}
When the $f_i$ all equal $f$ we simply write $\Lambda_q(f)$.

For a real parameter $H \geq 1$, we use $\mu_H : \Z \to [0,1]$ to represent the following normalised Fej\'er kernel
\begin{equation}\label{fejer}
\mu_H(h) := \recip{\floor{H}} \brac{1 - \frac{|h|}{\floor{H}}}_+ = \frac{(1_{[H]} * 1_{[H]} )(h)}{\floor{H}^2}.
\end{equation}
For a multidimensional vector $h \in \Z^d$ we write
\begin{equation}\label{multidim fejer}
\mu_H(h) := \mu_H(h_1)\dotsm \mu_H(h_d).
\end{equation}
We observe that this is a probability measure on $\Z^d$ with support in the box $(-H, H)^d$.

\subsection{An outline of our argument}\label{sec3}

\subsubsection{The density increment}\label{ss3.1}

Our proof proceeds via a density increment argument, the same method of proof used by Roth~\cite{RothCertainI} and  Gowers~\cite{GowersNewFour,GowersNew} to bound the size of sets lacking arithmetic progressions. In Gowers' formulation, if $A\subset[N]$ has density $\delta:=|A|/N$ and lacks (say) a $4$-term arithmetic progression, then either $N \ll \exp\exp\brac{\delta^{-O(1)}}$ or there exists a progression $P=a+q\cdot[N']$ of length $N'\gg  N^{\delta^{O(1)}}$ on which $A$ has increased density $|A\cap P|/|P|\geq \delta+\delta^{O(1)}$. 
Consider the rescaled version $A_1\subset[N']$ of $A\cap P$ defined by
\begin{equation}\label{3.??}
A':=\{n\in[N']:a+qn\in A\cap P\},
\end{equation}
and note that   $A'$ also lacks 4-term progressions. We then repeat this process with $A'$ in place of $A$. This iteration cannot continue indefinitely; indeed since the density cannot exceed one, the procedure must terminate in $O(\delta^{-O(1)})$ steps. The only explanation for termination is that the the length $N''$ of the interval at the final stage of our iteration is too short $N'' \ll \exp\exp\brac{\delta^{-O(1)}}$, and since $N'' \geq N^{\exp\brac{-\delta^{O(1)}}}$, this allows us to extract a bound on $\delta$.

The success of the above argument  relies crucially on the fact that 4-term arithmetic progressions are preserved under translation and scaling, and similarly the argument in~\cite{PrendivilleQuantitative} relies on the fact that arithmetic progressions with common difference equal to a perfect $d^{th}$ power are preserved under translation and scaling by a perfect $d^{th}$ power. These are very special properties lacked by the vast majority of polynomial progressions, including the nonlinear Roth configuration~\eqref{main config}.

Indeed, if $A\subset[N]$ has no nontrivial configurations of the form~\eqref{main config}, then the rescaled set $A'\subset[N']$ defined as in~\eqref{3.??} has no nontrivial configurations of the form $x$, $x+y$, $x+qy^2$. But if $q>N'$, then every subset of $[N']$ has this property because $x$ and $x+qy^2$ cannot both lie in $[N']$ when $y\neq 0$, and thus there is no hope of continuing the density increment argument in this case. In contrast, the largeness of $q>N'$ does not affect the arguments of \cite{RothCertainI, GowersNewFour,GowersNew}, because these papers consider progressions that are preserved under scaling by $q$ (or $q^d$ in the case of \cite{PrendivilleQuantitative}).

To deal with the poor behavior of the nonlinear Roth configuration under scaling, we prove a stronger density increment lemma that ensures that the arithmetic progression on which we find a density increment has very small step size. Our methods show that if $A\subset[N]$ has density $\delta:=|A|/N$ and lacks nontrivial configurations of the form~\eqref{main config}, then there exists a progression $P=a+q\cdot[N']$ with common difference $q\ll\delta^{-O(1)}$ and length $N'\gg\delta^{O(1)}N^{1/2}$ such that we have the density increment
\begin{equation}\label{eq3.1}
|A\cap P|/|P|\geq \delta+\delta^{O(1)}.
\end{equation}
Defining $A'\subset[N']$ to be the rescaled set as in~\eqref{3.??}, we thus see that $A'$ has increased density in $[N']$ and lacks nontrivial configurations of the form
\begin{equation}\label{eq3.2}
x,\ x+y,\ x+qy^2.
\end{equation}
The coefficient $q$ is sufficiently small that the methods employed to treat our original configuration~\eqref{main config} still apply to the new configuration~\eqref{eq3.2}, allowing us to prove a similar density increment result for sets lacking~\eqref{eq3.2}. We can thus continue the density increment iteration, which terminates in at most $O(\delta^{-O(1)})$ steps.  Such an argument yields a density bound of the form
$$
\delta \ll (\log\log N)^{-c}
$$
for some small absolute constant $c >0$. 

Our general density increment result is stated in Lemma~\ref{increment lemma} and concerns the configuration \eqref{eq3.2}.  It is a simple deduction from our inverse theorem (Theorem \ref{thm:inverse-zero}), or more precisely a generalisation (Theorem \ref{thm:inverse-one}) of our inverse theorem, extending from the nonlinear Roth configuration \eqref{main config} to its dilated analogue \eqref{eq3.2}.  In the remainder of this section we describe the ideas behind our inverse theorem, Theorem \ref{thm:inverse-zero}.

\subsubsection{Quantitative concatenation}\label{ss3.2}

To prove Theorem \ref{thm:inverse-zero}, we first prove that our counting operator
\begin{equation}\label{counting op 1}
\E_{x \in [N]} \E_{y \in [M]} f_0(x)f_1(x+y) f_2(x+y^2) 
\end{equation}
is controlled by the $U^5$-norm of $f_2$. The purpose of this subsection is to sketch how we do this with polynomial bounds.

By repeatedly applying both the Cauchy--Schwarz inequality and the van der Corput inequality, we show in \S\ref{pet} that, when $f_0,f_1,f_2:\Z\to\C$ are $1$-bounded functions supported in the interval $[N]$, largeness of the counting operator~\eqref{counting op 1} implies largeness of the sum
\begin{equation}\label{eq3.5}
\sum_{a,b\in[N^{1/2}]}\sum_{h_1,h_2,h_3\in[N^{1/2}]}\sum_x\Delta_{ah_1,bh_2,(a+b)h_3}f_2(x).
\end{equation}
This deduction is made following the PET induction scheme of Bergelson and Leibman \cite{BergelsonLeibmanPolynomial}. The gain in working with the counting operator \eqref{eq3.5} over \eqref{counting op 1} is that univariate polynomials such as $y^2$, whose image constitute a sparse set, have been replaced by bilinear forms such as $ah_1$, whose image is much denser

In \S\S\ref{arithmetic inverse}--\ref{concat sec}, we show that largeness of~\eqref{eq3.5} implies largeness of $\|f_2\|_{U^5}$.  If there were no dependence between the coefficients of the $h_i$ in~\eqref{eq3.5}, then it would be easy to bound~\eqref{eq3.5} in terms of $\|f_2\|_{U^3}$. 
We illustrate why this is the case for the sum
\begin{equation}\label{eq3.6}
\sum_{a,b,c\in[N^{1/2}]}\sum_{h_1,h_2,h_3\in[N^{1/2}]}\sum_x\Delta_{ah_1,bh_2,ch_3}f_2(x).
\end{equation}
The following fact is key, the formal version of which is Lemma \ref{densifying difference functions}. 
\begin{claim}\label{densify binary}
If $\displaystyle \sum_{a, h\in [N^{1/2}]} \sum_x \Delta_{ah} f(x)$ is large then so is $\displaystyle \sum_{k\in (-N, N)} \sum_x \Delta_{k} f(x)$.
\end{claim}
\begin{proof}[Sketch proof]
 Apply the Cauchy--Schwarz inequality to double the $a$ and $h$ variables, yielding a bound in terms of 
\begin{equation}\label{eq3.7}
\sum_{a,a'\in[N^{1/2}]}\sum_{h,h'\in[N^{1/2}]}\sum_x\Delta_{ah-a'h'}f(x).
\end{equation}
For a random choice of $a,a'\in[N^{1/2}]$, the progression $a\cdot[N^{1/2}]-a'\cdot[N^{1/2}]$ covers a large portion of the interval $(-N,N)$ relatively smoothly. One can make this intuition rigorous and thus deduce largeness of the sum
$
\sum_{k \in (-N, N)}\sum_x\Delta_{k}f(x).
$\end{proof}
Applying Claim \ref{densify binary} three times allows us to replace each of $ah_1$, $bh_2$  and $ch_3 $ in \eqref{eq3.6} with $k_1, k_2, k_3 \in (-N, N)$, yielding largeness of $\norm{f_2}_{U^3}$.  

The problem remains of how to handle the dependency between the differencing parameters in \eqref{eq3.5}. If we were not concerned with  quantitative bounds, we could apply a `concatenation' theorem of Tao and Ziegler~\cite[Theorem 1.24]{TaoZieglerConcatenation} to obtain largeness of the $U^9$-norm of $f_2$.  However, the qualitative nature of this argument means that it cannot be used to obtain bounds in the nonlinear Roth theorem. In its place we prove Theorem~\ref{global U5}, which is a special case of \cite[Theorem 1.24]{TaoZieglerConcatenation}, using a very different argument that gives polynomial bounds. We spend the remainder of this subsection sketching the argument.

We begin by viewing  \eqref{eq3.5} as the average
\begin{equation}\label{eq3.9}
\sum_{a, h_1 \in [N^{1/2}]} \norm{\Delta_{ah_1} f_2}_a,
\end{equation} 
where
\begin{equation}\label{eq3.8}
\|f\|_a^4:=\sum_{b\in[N^{1/2}]}\sum_{h_2,h_3\in[N^{1/2}]}\sum_x\Delta_{bh_2,(a+b)h_3}f(x)
\end{equation}
One can view this as an average of 2-dimensional Gowers box norms where, for fixed $b$, the inner sum corresponds to a box norm in the `directions' $b$ and $a+b$.
Note that if we could bound the quantity $\|\Delta_{a h_1}f_2\|_a$ in terms of the $U^4$-norm of $\Delta_{a h_1}f_2$ for many pairs $(a,h_1)$, then by Claim \ref{densify binary} we deduce largeness of the $U^5$-norm of $f_2$. We show that, on average, one can indeed control $\|\cdot\|_a$ in terms of $\|\cdot\|_{U^4}$, with polynomial bounds. The following can be extracted from the proof of (the more general) Theorem \ref{global U5}.
\begin{lemma}\label{lem3.4}
For each $a\in[N^{1/2}]$ let $f_a:\Z\to\C$ be a $1$-bounded function supported in the interval $[N]$. Suppose that
\[
\E_{a\in[N^{1/2}]}\|f_a\|_a^4\geq \delta \norm{1_{[N]}}_a^4.
\]
Then
\[
\E_{a\in[N^{1/2}]}\|f_a\|_{U^4}^{16}\gg \delta^{O(1)}\norm{1_{[N]}}_{U^4}^{16}.
\]
\end{lemma}

To finish this subsection, we briefly discuss the proof of this key lemma. For most choices of $a, b\in[N^{1/2}]$, the `directions' $b$ and $a+b$ of the box norm 
\begin{equation}\label{box norm}
\sum_{h_2,h_3\in[N^{1/2}]}\sum_x\Delta_{bh_2,(a+b)h_3}f_a(x)
\end{equation}
are close to `independent', in the sense that at least one of the directions $b$ and $a+b$ is large and together they have small greatest common divisor. The proof of Lemma~\ref{lem3.4} thus begins by viewing $\|\cdot\|_a$ as an average of box norms
\begin{equation}\label{real box norm}
\|f\|_{\square(X,Y)}^4:=\sum_{x_1,x_2\in X, y_1,y_2\in Y}f(x_1,y_1)\overline{f(x_1,y_2)f(x_2,y_1)}f(x_2,y_2).
\end{equation}
It is easy to show that largeness of $\|f\|_{\square(X,Y)}$ implies that $f$ correlates with a function of the form $(x,y)\mapsto l(x)r(y)$. We show, analogously, that provided $b$ and $a+b$ are not too small and have  greatest common divisor not too large, then largeness of the arithmetic box norm \eqref{box norm} implies that $f_a$ correlates with a product  $g_bh_{a+b}$ of 1-bounded functions, where $g_b$ is $b$-periodic and $h_{a+b}$ is almost periodic under shifts by integer multiples of $a+b$. As a consequence, for most $a\in[N^{1/2}]$, largeness of $\|f_a\|_{a}$ implies largeness of
\begin{equation}\label{eq3.10}
\sum_{b\in[N^{1/2}]}\sum_xf_a(x)g_b(x)h_{a+b}(x).
\end{equation}
In fact, an application of Cauchy--Schwarz allows us give an explicit description of $h_{a+b}$ in terms of $f_a$, namely we may take it to be of the form
\begin{equation}\label{h description}
h_{a+b}(x) = \E_{k\in [N^{1/2}]} f_a(x+(a+b)k)g_b(x+(a+b)k).
\end{equation}
This presentation makes apparent the almost periodicity of $h_{a+b}$.

\begin{claim}\label{hab claim}
Largeness of \eqref{eq3.10} implies that $\E_{b \in [N^{1/2}]}h_{a+b}$ has large $U^3$-norm.
\end{claim}

Let us first show why Claim \ref{hab claim} in turn implies that $f_a$ has large $U^4$-norm, completing our sketch proof of Lemma \ref{lem3.4}.  The expression \eqref{h description} and the triangle inequality for Gowers norms together imply that largeness of $\E_{b \in [N^{1/2}]}\norm{h_{a+b}}_{U^3}$ implies largeness of $\E_{b \in [N^{1/2}]}\norm{f_ag_b}_{U^3}$.  Utilising the $b$-periodicity of $g_b$ we have
\begin{equation}\label{fg product}
\norm{f_ag_b}_{U^3} = \E_{k \in [N^{1/2}]} \norm{f_a(\cdot)g_b(\cdot + bk)}_{U^3}.
\end{equation}
The product $f_a(\cdot)g_b(\cdot + bk)$ resembles a difference function in the direction $b$. Indeed the Gowers--Cauchy--Schwarz inequality (see \cite[Exercise 1.3.19]{TaoHigher}) shows that if \eqref{fg product} is large (on average over $b \in [N^{1/2}]$) then so is
$$
\E_{b,k \in [N^{1/2}]} \norm{\Delta_{bk}f_a}_{U^3}
$$
Largeness of $\norm{f_a}_{U^4}$ then follows from Claim \ref{densify binary}.

Finally we sketch the proof of Claim \ref{hab claim}.  The Cauchy--Schwarz inequality allows us to remove the weight $f_a(x)$ from \eqref{eq3.10} and deduce largeness of
$$
\sum_x\sum_{b,b' \in [N^{1/2}]} \overline{g_b(x)h_{a+b}(x)}g_{b'}(x)h_{a+b'}(x).
$$
Using the periodicity properties of $g_b$, $g_{b'}$ and $h_{a+b}$, this is approximately equal to
\begin{equation*}
\sum_x\sum_{\substack{b,b' \in [N^{1/2}]\\k_1, k_2, k_3 \in [N^{1/2}]}}\overline{g_b(x-bk_1)h_{a+b}(x-(a+b)k_2)}g_{b'}(x-b'k_3)h_{a+b'}(x).
\end{equation*}
Changing variables in $x$, we obtain largeness of the sum
\begin{multline*}
\sum_x\sum_{\substack{b,b' \in [N^{1/2}]\\k_1, k_2,k_3 \in [N^{1/2}]}}\overline{g_b(x+(a+b)k_2+b'k_3)h_{a+b}(x+bk_1+b'k_3)}\\
g_{b'}(x+bk_1+(a+b)k_2)h_{a+b'}(x+bk_1+(a+b)k_2+b'k_3).
\end{multline*}
The point here is that all but the last function have arguments depending on at most two of the bilinear forms $bk_1$, $(a+b)k_2$ and $b'k_1'$. This enables us to employ the Gowers--Cauchy--Schwarz inequality (in the form of Lemma \ref{box cauchy}) to deduce largeness of a sum similar to
$$
\sum_x\sum_{\substack{b,b'\in [N^{1/2}]\\ k_1, k_2, k_3\in [N^{1/2}]}}\Delta_{bk_1,\, (a+b)k_2,\, b'k_3}h_{a+b'}(x).
$$
The utility of this expression is that the directions of the differencing parameters are all `independent' of the direction of periodicity of $h_{a+b'}$.  Indeed the approximate $(a+b')$-periodicity of $h_{a+b'}$ means that one can replace $\Delta_y h_{a+b'}$ with $\E_k\Delta_{y + (a+b')k} h_{a+b'}$ at the cost of a small error. We thereby obtain largeness of
\begin{equation}\label{horrendous differences}
\sum_x\sum_{b, b' \in [N^{1/2}]}\sum_{\substack{k_1, k_2, k_3\in [N^{1/2}]\\ k_1', k_2', k_3'\in [N^{1/2}]}}\Delta_{bk_1 +(a+b')k_1',\, (a+b)k_2+ (a+b')k_2',\, b'k_3+(a+b')k_3'}h_{a+b'}(x).
\end{equation}
For a random triple $(a,b,b') \in [N^{1/2}]$ the greatest common divisor of the pairs $(b, a+b')$, $(a+b, a+b')$ and  $(b', a+b')$ are all small, and these are the pairs appearing in the differencing parameters of \eqref{horrendous differences}. The argument used to treat \eqref{eq3.7} may be therefore be employed  to replace \eqref{horrendous differences} with
$$
\sum_x\sum_{b'\in [N^{1/2}]}\sum_{k_1, k_2, k_3\in [N]}\Delta_{k_1 ,k_2,k_3}h_{a+b'}(x),
$$
and thereby yield Claim \ref{hab claim}.

\subsubsection{Degree lowering}\label{ss3.3}

After we have shown that the counting operator \eqref{counting op 1} is controlled by the $U^5$-norm of $f_2$, we carry out a `degree lowering' argument. This technique originated in the work~\cite{PelusePolynomial} in finite fields. The basic idea is that, under certain conditions, one can combine $U^s$-control with understanding of two-term progressions to deduce $U^{s-1}$-control. Repeating this gives a sequence of implications
\[
U^{5}\text{-control}\implies U^{4}\text{-control}\implies U^{3}\text{-control}\implies U^{2}\text{-control}\implies U^{1}\text{-control}.
\]
Despite the appearance of the $U^5$-norm, $U^4$-norm, and $U^3$-norm, the degree lowering argument, both in~\cite{PelusePolynomial} and here, does not require the $U^s$-inverse theorem for any $s\geq 3$. Instead it relies on Fourier analysis in the place of these inverse theorems.

As was mentioned in \S\ref{introduction}, adapting the degree lowering argument of~\cite{PelusePolynomial} to the integer setting requires several significant modifications. The first modification is that the $U^s$-control described above is control in terms of the $U^s$-norm of the dual function
\begin{equation}\label{sketch dual}
F(x):=\E_{y\in[N^{1/2}]}f_0(x-y^2)f_1(x+y-y^2).
\end{equation}
Thus, to begin the degree lowering argument, we must show that the counting operator \eqref{counting op 1} is controlled by the $U^5$-norm of the dual $\|F\|_{U^5}$. To do this, we use the fact that the counting operator is controlled by $\norm{f_2}_{U^5}$ together with a simple application of the Cauchy--Schwarz inequality, for details see \S\ref{inverse theorem section}. 

We illustrate our degree lowering procedure by sketching how $U^3$-control of the dual \eqref{sketch dual} implies $U^2$-control, starting from the assumption that
\[
\|F\|_{U^3}^8\geq\delta \norm{1_{[N]}}_{U^3}^8.
\]
Using the fact that $\|F\|_{U^3}^8=\sum_h\|\Delta_hF\|_{U^2}^4$ and applying the $U^2$-inverse theorem, we deduce the existence of a function $\phi:\Z\to\T$ such that, for at least $\gg\delta N$ choices of differencing parameter $h$, we have
\begin{equation}\label{eq3.11}
\left|\sum_{x\in[N]}\Delta_hF(x)e(\phi(h)x)\right|\gg\delta N.
\end{equation}
Note that if, in the above inequality, we could replace the function $\phi(h)$ by a constant $\beta\in\T$ not depending on $h$, then we could easily deduce largeness of $\|F\|_{U^2}$. Indeed, writing $g(h)$ for the conjugate phase of the sum inside absolute values, this would give 
$$
\sum_{x, h}\overline{g(h)}\overline{F(x+h)}F(x)e(\beta x)\gg\delta^{O(1)}N^3,
$$
and the usual argument\footnote{One can either use orthogonality and extraction of a large Fourier coefficient, as in the proof of Lemma \ref{U2 inverse}, or use two applications of Cauchy--Schwarz.} showing $U^2$-control of the equation $x+y = z$ implies that $\|F\|_{U^2}^4\gg\delta^{O(1)}\norm{1_{[N]}}_{U^2}$. It thus remains to show that such a $\beta$ exists.

Expanding the definition of the difference and dual functions in \eqref{eq3.11}, and using the Cauchy--Schwarz inequality (as is done in greater generality in the proof of Lemma~\ref{dual difference interchange}), one can show that there exists $h'$ such that for many $h$ satisfying \eqref{eq3.11} we have
\[
\left|\sum_x\sum_{y\in[N^{1/2}]}\Delta_{h-h'}f_0(x)\Delta_{h-h'}f_1(x+y)e([\phi(h)-\phi(h')][x+y^2])\right|\gg \delta^{O(1)}N^{3/2}
\]
Further application of Cauchy--Schwarz allows us to remove the difference functions from the above inequality and deduce largeness of the exponential sum
\[
\sum_{z\in[N^{1/2}]}\left|\sum_{y\in[N^{1/2}]}e(2\sqbrac{\phi(h)-\phi(h')} yz)\right|.
\]
Summing the inner geometric progression and using a Vinogradov-type lemma then shows that $\phi(h)-\phi(h')$ is major arc. There are very few major arcs, so the pigeonhole principle gives the existence of $\beta_0\in\T$ such that $\phi(h)-\phi(h')$ is very close to $\beta_0$ for many $h\in(-N,N)$ that also satisfy~\eqref{eq3.11}. We may therefore take $\beta=\beta_0+\phi(h')$ in the argument following \eqref{eq3.11}.

\section{The density increment}\label{increment sec}
In this section we prove Theorem \ref{main} using the following lemma, which is derived from our inverse theorem in \S\ref{increment proof sec}.
\begin{lemma}[Density increment]\label{increment lemma}
Suppose that $A \subset [N]$ satisfies $|A| \geq \delta N$ and lacks the configuration 
\begin{equation}\label{q config}
x,\ x+y, \ x+qy^2 \qquad (y \neq 0).
\end{equation} 
Then either $N \ll q^3\delta^{-O(1)}$ or there exists $q' \ll \delta^{-O(1)}$ and $N' \gg \delta^{O(1)}q^{-3/2}N^{1/2} $
 such that
\begin{equation}\label{increment}
 |A \cap (a + qq'\cdot[N'])| \geq \brac{\delta + \Omega\brac{\delta^{O(1)}}} N'.
\end{equation}
\end{lemma}

\begin{proof}[Proof of Theorem \ref{main} given Lemma \ref{increment lemma}]
Note first that if $A$ lacks the configuration \eqref{q config}, then the set
\[
\{x :a+qq'x\in A\},
\]
lacks configurations of the form
\[
x,\ x+y,\ x+q^2q'y^2 \qquad (y \neq 0).
\]

Let $A \subset [N]$ have size $\delta N$ and lack \eqref{main config}.  Setting $A_0 := A$, $N_0 := N$ and $q_0 = 1$, let us suppose we have a  sequence of tuples $(A_i, N_i, q_i)$ for $i = 0, 1, \dots, n$ which each satisfy the following:
\begin{enumerate}[(i)]
\item  $A_i$ lacks configurations of the form 
$$
x,\ x+y,\ x+q_0^{2^i} q_1^{2^{i-1}}\dotsm q_{i-1}^2 q_i y^2 \qquad (y \neq 0).
$$
\item\label{qi upper bound}  $q_i \ll \delta^{-O(1)}$;
\item  $A_i \subset [N_i]$ and for $i \geq 1$ we have 
$$
\frac{|A_i|}{N_i} \geq \delta + \Omega\brac{i\delta^{O(1)}};
$$
\item\label{length lower bound}  for $i \geq 1$ we have the lower bound 
$$
N_i \gg \delta^{O(1)} \brac{q_0^{2^{i-1}}\dotsm q_{i-1}}^{-3/2}N_{i-1}^{1/2}.
$$
\end{enumerate}  

By Lemma \ref{increment lemma}, at stage $n$ we either have
\begin{equation}\label{termination condition}
N_n \ll \brac{q_0^{2^n} q_1^{2^{n-1}}\dotsm q_{n-1}^2 q_n}^3\delta^{-O(1)} 
\end{equation}
and the process terminates, or we obtain $(A_{n+1}, N_{n+1}, q_{n+1})$ satisfying conditions (i)--(iv).

The density of $A_n$ on $[N_n]$ cannot exceed 1, so the process must terminate at some $n \ll \delta^{-O(1)}$.  At the point of termination, the  smallness assumption \eqref{termination condition} must hold, so that
\begin{equation*}\label{N_i upper bound}
N_n \leq O(1/\delta)^{O(2^n)} \leq \exp\exp\brac{O\brac{\delta^{-O(1)}}}.
\end{equation*}
On the other hand, iteratively applying the lower bound \eqref{length lower bound}, we have 
\begin{equation*}\label{N_i lower bound}
\begin{split}
N_n  & \geq \frac{N_{n-1}^{1/2}}{\brac{q_0^{2^{n-1}}\dotsm q_{n-1}}^{3/2}\delta^{-O(1)}}\\
&  \geq N^{1/2^n}\sqbrac{\brac{q_0^{2^{n-1}}\dotsm q_{n-1}}^{3/2}\delta^{-O(1)}}^{-(1 + \recip{2} + \recip{4} + \dots + 2^{1-n}) } \\
& \geq  \exp\brac{\log N \exp\brac{-O\brac{\delta^{-O(1)}}}}\exp\exp\brac{-O\brac{\delta^{-O(1)}}}.
\end{split}
\end{equation*}
Taking logarithms and comparing upper and lower bounds for $N_n$ yields the bound claimed in Theorem \ref{main}.
\end{proof}

\section{PET induction}\label{pet}

We prove Lemma \ref{increment lemma} over the course of \S\S\ref{pet}--\ref{increment proof sec}. We begin in \S\S\ref{pet}--\ref{concat sec} by showing how the counting operator $\Lambda_q(f_0, f_1, f_2)$, as defined in \eqref{counting op}, is controlled by the $U^5$-norm of $f_2$.  This argument starts with the PET induction scheme of Bergelson--Leibman \cite{BergelsonLeibmanPolynomial}, which in some sense `linearises' a polynomial progression, replacing univariate polynomials such as $y^2$ with bilinear forms $ah$.  The outcome of this procedure is Lemma \ref{linearisation}.

For the following, we recall our definition \eqref{fejer} of the Fej\'er kernel $\mu_H$ and our definition  \eqref{eq:diff} of the difference operator $\Delta_h$.
\begin{lemma}[van der Corput inequality]\label{vdc}
Let $f:\Z\to\C$ be 1-bounded,  $J$ an interval of at most $M$ integers and $ H \geq 1$. For each $h$ there exists an interval $J(h) \subset J$ such that
\[
\biggabs{\sum_{y\in J}f(y)}^2\leq \brac{M+H}\sum_h\mu_H(h)\sum_{y\in J(h)}\Delta_h f(y) .
\]
\end{lemma}

\begin{proof}
This is standard, see for instance \cite[Lemma 3.1]{PrendivilleQuantitative}.\end{proof}

The next result uses the multivariate Fej\'er kernel $\mu_H(h)$ as defined in \eqref{multidim fejer}.

\begin{lemma}[Difference functions control linear configurations]\label{difference control}
 Let $f_{i}:\Z\to\C$ be $1$-bounded functions each with support of size at most $N$. Let $J\subset \Z$ be an interval of size at most $M$.
Then for any $a,b \in \Z$ and $H\geq 1$ we have
\begin{multline}\label{difference control ineq}
\biggabs{\sum_{x}\sum_{y\in J}f_{0}(x)f_{1}(x+ay)f_{2}(x+by)f_{3}(x+(a+b)y)}^8\\
 \leq N^7(M+H)^8\sum_h\mu_{H}(h)\sum_{x}\Delta_{ah_1,bh_2,(a+b)h_3}f_{3}(x).
\end{multline}
\end{lemma}

\begin{proof}
Applying Cauchy-Schwarz in the $x$ variable gives
\begin{multline*}
\biggabs{\sum_{x }\sum_{y\in J}f_{0}(x)f_{1}(x+ay)f_{2}(x+by)f_{3}(x+(a+b)y)}^2\\
 \leq
N\sum_{x}\bigg|\sum_{y\in J}f_1(x+ay)f_2(x+by)
f_3(x+(a+b)y)\bigg|^2.
\end{multline*}
 Bounding the inner sum using van der Corput's inequality (Lemma \ref{vdc}) and making the change of variables $x \mapsto x-ay$ (valid since $x$ is ranging over $\Z$), the latter is at most  
$$
N(M+H)\sum_{h_1}\mu_H(h_1)\sum_{x }\sum_{y \in J(h_1)}\Delta_{ah_1}f_1(x) \Delta_{bh_1} f_2(x+(b-a)y) \Delta_{(a+b)h_1}f_3(x+by),
$$
for some intervals $J(h_1) \subset J$.

Making use of the fact that $\mu_H$ is a probability measure, we repeat the procedure of applying Cauchy--Schwarz, van der Corput then a change of variables, to deduce that 
\begin{multline*}
\biggabs{\sum_{x}\sum_{y\in J}f_{0}(x)f_{1}(x+ay)f_{2}(x+by)f_{3}(x+(a+b)y)}^4\\
 \leq N^3(M+H)^3 \sum_{h_1, h_2}\mu_H(h_1,h_2)\sum_{x}\sum_{y\in J(h_1, h_2)} \Delta_{bh_1, (b-a)h_2} f_2(x) \Delta_{(a+b)h_1, bh_2}f_3(x+ay),
\end{multline*}
for some intervals $J(h_1, h_2) \subset J$. A final iteration of the same procedure then yields \eqref{difference control ineq}.
\end{proof}

Before embarking on the following, we remind the reader of our conventions \eqref{M def} and \eqref{counting op} regarding $M$ and $\Lambda_q$.
\begin{lemma}[Linearisation]\label{linearisation}
Let $f_i:\Z\to\C$ be $1$-bounded functions with support in $[N]$. Then for any  $1 \leq H\leq  M$ we have
\begin{equation}\label{linearised ineq}
\abs{\recip{NM}\Lambda_q(f_0, f_1, f_2) }^{32} \ll
 \sum_{a,b, h}\mu_{M}(a)\mu_{M}(b)\mu_H(h)\E_{x\in [N]} \Delta_{2q(a+b)h_1,\, 2qbh_2,\, 2qah_3}f_2(x).
\end{equation}
\end{lemma}
\begin{proof}
We repeat the procedure given in the proof of Lemma \ref{difference control}, applying Cauchy-Schwarz, followed by van der Corput's inequality and a change of variables. We note that for $y\in \N$, if $f_0(x)f_1(x+y)f_2(x+qy^2) \neq 0$ for some $x$, then $y \in [M]$, since $qy^2 = (x+qy^2) - x \in [N]-[N]$.  A first application of this procedure gives
\begin{multline*}
\abs{\recip{NM}\Lambda_q(f_0, f_1, f_2) }^2 \leq\\
 \frac{2}{NM} \sum_{a}\mu_M(a) \sum_{x} \sum_{y\in J(a)} \Delta_{a}f_1(x)f_2\bigbrac{x+qy^2-y}\overline{f_2\bigbrac{x+q(y+a)^2-y}},
\end{multline*}
for some intervals $J(a) \subset [M]$.  A second application then gives
\begin{multline*}
\abs{\Lambda_q(f_0, f_1, f_2) }^4 \ll \recip{NM}
 \sum_{a,b}\mu_M(a)\mu_M(b) \sum_{x} \sum_{y\in J(a,b)} f_2(x)\overline{f_2\bigbrac{x+2qay + qa^2}}\\ \overline{f_2\bigbrac{x+2qby + qb^2-b}}f_2\bigbrac{x+2q(a+b)y + q(a+b)^2-b},
\end{multline*}
for some intervals $J(a,b) \subset [M]$.
Applying Lemma \ref{difference control} to bound the inner sum over $x$ and $y$, we obtain \eqref{linearised ineq}. 
\end{proof}

\section{An inverse theorem for the arithmetic box norm}\label{arithmetic inverse}

The objective in this section is to characterise those 1-bounded functions $f : \Z \to \C$ with support in $[N]$ for which the following quantity is large
\begin{equation}\label{b}
\sum_{h,x}\mu_{H}(h)\Delta_{ah_1,bh_2}f(x).
\end{equation}
One can think of this as an arithmetic analogue of the two-dimensional `box norm' \eqref{real box norm}, with differencing parameters pointing in the `directions' $a$ and $b$.  In our eventual application we are able to ensure that $a$ and $b$ are a generic pair of integers from the interval $[N^{1/2}]$.  In particular, at least one of them has size proportional to $N^{1/2}$ and their highest common factor is small.  One may think of this as a proxy for linear independence. 

We begin by characterising largeness of \eqref{b} when the directions are coprime.

\begin{lemma}[Inverse theorem for the arithmetic box norm]
Let $a,b$ be positive integers with $\gcd(a,b) = 1$. 
Suppose that $f:\Z\to\C$ is $1$-bounded with support in the interval $[N]$ and satisfies
\begin{equation}\label{arith1}
\sum_{h,x}\mu_{H}(h)\Delta_{ah_1,bh_2}f(x)\geq \delta N.
\end{equation}
Then there exist 1-bounded functions $g, h:\Z\to\C$ such that
\begin{itemize}
\item $g$ is $a$-periodic, in the sense that $g(x+a) = g(x)$ for all $x$;
\item $h$ is approximately $b$-periodic, in the sense that for any $\eps > 0$ we have
$$
\hash\set{x \in [N] : h(x+by) \neq h(x) \text{ for some } |y| \leq \eps N/b}  \leq \brac{1+\tfrac{2\eps N}{b}}\brac{1 + \tfrac{N}{a}};
$$
\end{itemize}
and furthermore
\begin{equation}\label{1st box inverse}
\biggabs{\sum_{x}f(x)g(x)h(x)}\geq \delta \floor{H}^2  - 2\brac{\tfrac{H}{a} + \tfrac{Hb}{N}}\floor{H}^2.
\end{equation}
\end{lemma}
\begin{remark}
In parsing the above inequalities, it may be helpful to keep in mind that in our application $a$, $b$ and $H$ are of order $\sqrt{N}$, with $H$ smaller than $\delta a$, in which case the lower bound in \eqref{1st box inverse} becomes $\Omega(\delta H^2)$.
\end{remark}
\begin{proof}
The majority of our proof is concerned with manipulating \eqref{arith1} until we can interpret it as a genuine box norm \eqref{real box norm}, and thereby apply the box norm inverse theorem. The essential observation is that, since $\gcd(a,b)=1$, every integer $x$ can be uniquely represented in the form
\[
x=ay+bz\qquad (y\in\Z,\ z\in[a]).
\]
We note that if $x \in [N]$ then the constraint on $z$ forces $y$ to lie in the range $-b < y < N/a$.

Defining $F:\Z\times\Z\to\C$ by $F(y,z):=f(ay+bz)$, the left-hand side of \eqref{arith1} becomes
$$
\sum_{y,y' \in \Z} \sum_{\substack{z \in [a]\\ z' \in \Z}}F(y, z)\overline{F(y', z)}
\overline{F(y, z')}F(y', z')\mu_H(y'-y)\mu_H(z'-z).
$$
If $z'$ and $z$ contribute to the above sum then
$
z' \in z + (-H, H) \subset (-H+1, a +H).
$
Hence we can restrict the range of summation of $z'$ to $[a]$, at the cost of perturbing the sum by at most
$
2\floor{H}(\frac{N}{a} + b).
$  
It follows that
\begin{multline*}
\biggabs{\sum_{y,y'}\sum_{z, z'\in [a]} F(y, z)\overline{F(y', z)}
\overline{F(y, z')}F(y', z')\mu_H(y'-y)\mu_H(z'-z)}\\ \geq \delta N - 2\floor{H}\brac{\tfrac{N}{a} + b}.
\end{multline*}

We remove the Fej\'er kernels by Fourier expansion:
\begin{multline*}
\sum_{\substack{y,y' \\ z,z'\in[a]}}F(y,z)\overline{F(y',z)F(y,z')}F(y',z')\mu_H(y'-y)\mu_H(z'-z) =\\
\int_{\T^2}\sum_{\substack{y,y' \\ z,z'\in[a]}}F(y,z)\overline{F(y',z)F(y,z')}F(y',z')\hat{\mu}_H(\alpha)\hat{\mu}_H(\beta)e(\alpha(y'-y)+\beta(z'-z))\intd\alpha\intd\beta \\
\leq\left(\int_{\T}|\hat{\mu}_H(\alpha)|\intd\alpha\right)^2\sup_{\alpha, \beta\in\T}\biggabs{\sum_{\substack{y,y' \\ z,z'\in[a]}}F(y,z)\overline{F_2(y',z)F_3(y,z')}F_4(y',z')},
\end{multline*}
where $F_2(y',z):=F(y',z)e(\beta z)$, $F_3(y,z'):=F(y,z')e(\alpha y)$, and $F_4(y',z')$ $:=F(y',z')e(\alpha y'+\beta z')$.

We observe that $\hat{\mu}_H(\alpha)=|\hat{1}_{[H]}(\alpha)|^2/\floor{H}^2$, which implies that $\int_\T|\hat{\mu}(\alpha)|d\alpha=\floor{H}^{-1}$. Therefore
\begin{equation}\label{arith2}
\biggabs{\sum_{\substack{y,y' \\ z,z'\in[a]}}F(y,z)\overline{F_2(y',z)F_3(y,z')}F_4(y',z')} \geq \delta \floor{H}^2 N - 2\floor{H}^3\brac{\tfrac{N}{a} + b},
\end{equation}
for $1$-bounded functions $F_i:\Z\times[a]\to\C$ of the form $F_i(y,z)=f(ay+bz)e(\alpha_1y+\alpha_2z)$. Since $f$ is supported on $[N]$, there are at most $N$ pairs $(y',z')\in\Z\times[a]$ for which $F_4(y',z')\neq 0$. Thus, by pigeonholing in $y'$ and $z'$ in~(\ref{arith2}) and setting $L(y):=\overline{F_3(y,z')}$ and $R(z):=\overline{F_2(y',z)}F_4(y',z')$, we get that
\[
\biggabs{\sum_{y}\sum_{z\in[a]}F(y,z)L(y)R(z)}\geq \delta \floor{H}^2 - 2\floor{H}^3\brac{\tfrac{1}{a} + \tfrac{b}{N}}.
\]

For each $x\in\Z$, define $l(x)\in\Z$ and $r(x)\in[a]$ by $x=al(x)+br(x)$, and set $g(x):=R\circ r(x)$ and $h(x):=L\circ l(x)$. Then
it remains to check the invariance properties of $g$ and $h$. To see that $g(x)=g(x+ay)$ for all $x,y\in\Z$, just note that $r(x)=r(x+ay)$ for every $x,y\in\Z$. 

Finally we establish that, for most $x\in[N]$, we have $h(x)=h(x+bz)$ when $|z|\leq \ve N/b$.  First note that $l(x)=l(x+bz)$ whenever $\ve N/b < r(x)\leq a-\ve N/b$. Hence for this to fail, $x$ must lie in one of at most $1+2\ve N/b$ congruence classes modulo $a$. The number of such $x$ lying in the interval $[N]$ is at most
\[
\brac{1+\frac{2\eps N}{b}}\brac{1 + \frac{N}{a}}.
\]

\end{proof}

The lemma also yields a result in the situation in which $\gcd(a,b)>1$.  In proving this we take the opportunity to smooth out the $b$-invariance of the $h$ function, whilst also giving an explicit description of $h$ in terms of $f$.  More concretely, we replace $h$ with a projection of $fg$ onto cosets of $b\cdot \Z$.
\begin{lemma}\label{arithcor} 
There exists an absolute constant $c>0$ such that on assuming  $1 \leq H \leq c\delta^3 N^{1/2}$ and $1 \leq K \leq c\delta^2 H^2 N^{-1/2}$ the following holds. Let $a,b\in [-N^{1/2},N^{1/2}]$ with $\gcd(a,b)\leq \delta^{-1}$ and $|a| \geq \delta N^{1/2}$.  Suppose that $f:\Z\to\C$ is $1$-bounded, supported on the interval $[N]$, and satisfies
\begin{equation*}\label{arith1}
\biggabs{\sum_{h, x}\mu_{H}(h)\Delta_{ah_1,bh_2}f(x)}\geq\delta N.
\end{equation*}
Then there exists a 1-bounded $a$-periodic function $g$ such that 
\begin{equation}\label{improved correlation}
\sum_{x}  f(x) g(x)\sum_k\mu_K(k) \overline{f(x+bk) g(x+bk)} \gg \delta^2 H^4/N.
\end{equation}
\end{lemma}

\begin{proof}
Since \eqref{arith1} and \eqref{improved correlation} are invariant under the transformations $a \mapsto -a$ and $b\mapsto -b$, we may assume that both $a$ and $b$ are positive.  
Set $q := \gcd(a,b) \leq \delta^{-1}$.  For each $u\in [q]$, define a $1$-bounded function $f_u:\Z\to\C$ by $f_u(x):=f(u +qx)$, and let $I_u := \set{x : u + qx \in [N]}$ denote the interval on which $f_u$ is supported. By the pigeon-hole principle, for some $u$ we have
\[
\sum_{x,h_1, h_2}\mu_{H}(h_1)\mu_{H}(h_2)\Delta_{\f{a}{q}h_1,\f{b}{q} h_2}f_u(x)\geq\delta |I_u|.
\]

Note that $\gcd(a/q,b/q)=1$, so by the previous lemma, there exist 1-bounded functions $g_u,h_u:\Z\to\C$ such that
\[
\biggabs{\sum_xf_u(x)g_u(x)h_u(x)}\geq \delta \floor{H}^2 - 2\brac{\tfrac{Hq}{a} + \tfrac{Hb}{q|I_u|}}\floor{H}^2 \gg \delta H^2.
\]
Furthermore, $g_u$ is $(a/q)$-periodic and
\begin{multline*}
\#\set{x\in I_u:h_u(x)\neq h_u(x+yb/q)\text{ for some }|y| \leq \eps |I_u|q/b}\\ \leq \brac{1+\tfrac{2q\eps |I_u|}{b}}\brac{1 + \tfrac{q|I_u|}{a}} \ll \tfrac{N}{a} + \tfrac{\eps N^2}{ab}.
\end{multline*}

Defining $g_{u'}$ and $h_{u'}$ to be identically zero when $u'\neq u$, we set $g(u'+qx):=g_{u'}(x)$ and $h(u'+qx):=h_{u'}(x)$.  One can then check that $g$ is $a$-invariant, that
$$
\biggabs{\sum_xf(x)g(x)h(x)} \gg \delta H^2,
$$
and that
$$
\#\set{x\in[N]:h(x)\neq h(x+by)\text{ for some }|y| \leq \eps N/b}\ll \tfrac{N}{a} + \tfrac{\eps N^2}{ab}.
$$

We may use the latter property to show that, provided $K \geq 1$, we have
$$
\biggabs{\sum_{x} f(x) g(x) h(x) - \sum_{x} h(x) \E_{y \in [K]} g(x+by) f(x + by) } \ll \tfrac{NK}{a} .
$$
Provided that  $K \leq c\delta^2 H^2 N^{-1/2}$ we deduce that
$$
\biggabs{\sum_xh(x) \E_{y \in [K]} g(x+by) f(x + by)} \gg \delta H^2.
$$
One can check that, as a function of $x$, the inner expectation is 1-bounded with support in $[-2N, 2N]$.  Applying the Cauchy--Schwarz inequality and changing variables then gives \eqref{improved correlation}.
\end{proof}
Finally we observe that a function of the form 
\begin{equation}\label{h def}
h(x) := \sum_k \mu_K(k) f(x+bk)
\end{equation} 
has nice $b$-periodicity properties.

\begin{lemma}\label{h lipschitz}
If $h$ is defined as in \eqref{h def} for some 1-bounded $f$, then $h$ is $O(K^{-1})$-Lipschitz along $b \cdot \Z$, in that for any $x, y\in \Z$ we have $h(x+by) = h(x) + O(|y|/K)$.
\end{lemma}

\begin{proof} 
Recalling the definition \eqref{fejer}, note that $\mu_K$ is $\floor{K}^{-2}$-Lipschitz, in that $|\mu_K(k+y) - \mu_K(k)| \leq |y|/\floor{K}^2$ for all $k, y \in \Z$.  Hence, for $|y| \leq K$, a change of variables gives
$$
|h(x+by) - h(x)|  \leq  \sum_k |\mu_K(k-y) - \mu_K(k)| \leq \frac{|y|}{\floor{K}^2}\sum_{|k| < 2\floor{K}} 1.
$$
\end{proof}

\section{Quantitative concatenation: control by a global Gowers norm}\label{concat sec}

The endpoint of this section is to show how our counting operator \eqref{counting op} is controlled by the $U^5$-norm.  We begin with four technical lemmas.  The first says that
convolving Fej\'er kernels along progressions of coprime common difference covers a substantial portion of an interval in a somewhat regular manner, a fact that can be interpreted Fourier analytically in the following.

\begin{lemma}\label{L1 Fourier bound}
Let $K, L \geq 1$ and let $a, b$ be integers satisfying $|a| \geq \delta L$, $|b| \geq \delta K$ and $\gcd(a, b) \leq \delta^{-1}$.  Then
$$
\int_\T \bigabs{\widehat{\mu}_K(a\beta)}\bigabs{\widehat{\mu}_L(b\beta)}\intd\beta \ll \frac{\delta^{-4}}{\floor{K}\floor{L}} .
$$
\end{lemma}

\begin{proof}
Taking complex conjugates inside the absolute values, we may assume that $a$ and $b$ are positive. Expanding Fourier transforms, one can check that
\begin{multline*}
\int_\T \bigabs{\widehat{\mu}_K(a\beta)}\bigabs{\widehat{\mu}_L(b\beta)}\intd\beta\\
 = \floor{K}^{-2} \floor{L}^{-2}\hash\biggset{(x, y) \in [K]^{2}\times[L]^{2} : a(x_1 - x_2) = b(y_1 - y_2)}.
\end{multline*}
Writing $d := \gcd(a,b)$, the number of solutions to the equation is at most
$$
\floor{K}\floor{ L} \brac{\tfrac{\floor{K}}{b/d} + 1}\brac{\tfrac{\floor{L}}{a/d} + 1}.
$$
\end{proof}

Our next lemma allows us to discard pairs of integers $a,b$ which are not sufficiently coprime. We exploit this repeatedly.

\begin{lemma}\label{gcd}
For fixed integers $|a_1|, |a_2| \leq M$.  The number of pairs $(b,c)$ of integers $ |b|,|c| \leq M$ such that $\gcd(a_1+b,a_2+c)> \delta^{-1}$ is $\ll \delta M^2$.
\end{lemma}
\begin{proof}
Notice that if $d = \gcd(a_1+b, a_2+c)$ then $d \leq 2M$.  Hence
\begin{align*}
\sum_{\substack{ |b|,|c| \leq M\\ \gcd(a_1+b,a_2+c) > \delta^{-1}}} 1  \leq \sum_{\delta^{-1} < d \leq 2M}\ \biggbrac{\ \sum_{ |m| \leq 2M,\ d \mid m} 1}^2
 &\leq  \sum_{\delta^{-1} < d \leq 2M} \brac{\frac{4M}{d} + 1}^2\\ & \ll M^2 \sum_{d> \delta^{-1}} \recip{d^2}
  \ll \delta M^2 .
\end{align*}
\end{proof}

The following lemma says that, as $a$ and $h$ range over $[N^{1/2}]$, the difference function $\Delta_{ah} f$ behaves like $\Delta_k f$ with $k \in [N]$, at least on average.

\begin{lemma}\label{densifying difference functions}
Let $f : \Z \to \C$ be a 1-bounded function with support in $[N]$.  Suppose that $\delta N^{1/2} \leq H \leq N^{1/2}$ and 
$$
 \E_{|a| \leq N^{1/2}}\sum_h \mu_H(h) \norm{\Delta_{ah} f }_{U^s}^{2^s} \geq \delta  \norm{1_{[N]}}_{U^s}^{2^s}.
$$ 
Then
$$
\norm{f}_{U^{s+1}}^{2^{s+1}} \gg \delta^{12}  \norm{1_{[N]}}_{U^{s+1}}^{2^{s+1}}
$$
\end{lemma}

\begin{proof}  Expanding the definition of the $U^s$-norm 
\begin{multline*}
\E_{|a| \leq N^{1/2}}\sum_h \mu_H(h) \norm{\Delta_{ah} f }_{U^s}^{2^s}\\ = \sum_{h_1, \dots, h_s, x} \overline{\Delta_{h_1, \dots, h_s}f(x)} \E_{|a| \leq N^{1/2}}\sum_h \mu_H(h) \Delta_{h_1, \dots, h_s} f(x+ ah).
\end{multline*}
Employing the Cauchy--Schwarz inequality to double the $a$ and $h$ variables gives
$$
\E_{|a|, |a'| \leq N^{1/2}}\sum_{h_i} \sum_x \sum_{h,h'} \mu_H(h)\mu_H(h')  \Delta_{h_1, \dots , h_s, ah - a'h'}f(x) \gg \delta^2  N^{s+1}. 
$$

By Lemma \ref{gcd} and the pigeon-hole principle, we deduce the existence of $|a|, |a'| \gg \delta^2N^{1/2}$ with $\gcd(a, a') \ll \delta^{-2}$  such that
$$
\sum_{h_i} \sum_x \sum_{h,h'} \mu_H(h)\mu_H(h')  \Delta_{h_1, \dots , h_s, ah - a'h'}f(x) \gg \delta^2 N^{s+1} .
$$
By Fourier inversion and extraction of a large Fourier coefficient, there exists $\alpha \in \T$ such that the right-hand side above is at most
$$
  \int_\T \abs{\widehat{\mu}_H(a\beta)}\abs{\widehat{\mu}_H(a'\beta)}\intd\beta\biggabs{\sum_{h_i} \sum_x  \Delta_{h_1, \dots , h_s, h_{s+1}}f(x)e(\alpha h_{s+1})}.
$$
The result follows on employing Lemma \ref{L1 Fourier bound} and Lemma \ref{phase invariance}.
\end{proof}

We now prove a similar lemma, but with $\Delta_{ah} f$  replaced by  $fg_a$ where $g_a$ is $a$-periodic.  The moral is that these are similar quantities (on average).
\begin{lemma}\label{difference vs periodic product}
Let $f, g_a : \Z \to \C$ be 1-bounded functions such that $g_a$ is $a$-periodic and $\supp(f) \subset[N]$.  Suppose that 
$$
\E_{|a| \leq N^{1/2}} \norm{f g_a}_{U^s}^{2^s} \geq \delta  \norm{1_{[N]}}_{U^s}^{2^s}.
$$ 
Then
$$
\norm{f}_{U^{s+1}}^{2^{s+1}} \gg \delta^{24}  \norm{1_{[N]}}_{U^{s+1}}^{2^{s+1}}
$$
\end{lemma}

\begin{proof}
Fix $|a| \leq N^{1/2}$. By the periodicity of $g_a$ and a change of variables, we have
$$
\sum_{h_i} \sum_x \Delta_{h_1, \dots, h_s}g_a(x)\Delta_{h_1,\dots, h_s}f(x) = \sum_{h_i} \sum_x \Delta_{h_1, \dots, h_s}g_a(x)\E_{y \in [N^{1/2}]}\Delta_{h_1, \dots, h_s}f(x + ay).
$$
Notice that the sum over $x$ is non-zero only if $|x|, |h_i| < N$, hence by Cauchy--Schwarz and a change of variables
\begin{align*}
\biggbrac{\E_{|a|\leq N^{1/2}}\norm{fg_a}_{U^s}^{2^s}}^2 & \ll N^{s+1}  \E_{|a| \leq N^{1/2}}\sum_{h_i} \sum_x \sum_y \mu_{N^{1/2}}(y) \Delta_{h_1, \dots, h_s, ay}f(x)\\
& =N^{s+1} \E_{|a| \leq N^{1/2}}\sum_y \mu_{N^{1/2}}(y) \norm{\Delta_{ay}f}_{U^{s}}^{2^s}
\end{align*}
The result follows on employing Lemma \ref{densifying difference functions}.
\end{proof}

We are now ready to give the technical heart of this section.  The (somewhat lengthy) assumptions come from our eventual application of Lemma \ref{arithcor}.
\begin{lemma}\label{hb lemma}
Fix $a \in \Z$ and let $\delta N^{1/2} \leq K \leq N^{1/2}$.  For each $b $ let $f, g_b, h_b : \Z \to \C$ be 1-bounded functions such that $\supp(f), \supp(h_b) \subset [N]$ and where $g_b$ is $b$-periodic.  Set
\begin{equation}\label{eq:hb-def}
\tilde{h}_b(x) := \sum_k \mu_K(k) h_{b}(x+(a+b)k)
\end{equation}
and suppose that
$$
\sum_{\substack{\delta \sqrt{N}\leq |b| \leq\sqrt{N}\\ \gcd(a,b) \leq \delta^{-1}}}  \sum_x f(x)g_b(x)\tilde{h}_b(x) \geq \delta N^{3/2}.
$$
Then 
$$
\E_{ |b| \leq N^{1/2}} \bignorm{h_b}_{U^3}^8 \gg \delta^{O(1)}  \norm{1_{[N]}}_{U^3}^8.
$$
\end{lemma}

\begin{proof}
We apply Cauchy--Schwarz to remove the weight $f(x)$ and double the $b$ variable, yielding
$$
\sum_{\substack{\delta \sqrt{N}\leq |b|,|b'| \leq\sqrt{N}\\ \gcd(a,b),\gcd(a,b') \leq \delta^{-1} }} \sum_x g_b(x) \tilde{h}_{b}(x)\overline{g_{b'}(x) \tilde{h}_{b'}(x)} \geq \delta^2  N^2.
$$
Employing Lemma \ref{gcd}, we may discard those $b,{b'}$ for which one of $\gcd(b', a+{b})$ or $\gcd(a+b', a+{b})$ is greater than $C\delta^{-2}$.  We may also discard those $b,b'$ for which either $|a+b|\leq c\delta^2 \sqrt{N}$ or $|a+b'|\leq c\delta^2 \sqrt{N}$. On combining this with the popularity principle, we deduce the existence of $\mathcal{B} \subset [- N^{1/2}, N^{1/2}]$ of size $|\mathcal{B}| \gg \delta^2 N^{1/2}$ such that for each $b \in \mathcal{B}$ there exists $|b'| \leq \sqrt{N} $ with all of $|b|, |b'|, |a+b|, |a+b'|$ $\gg \delta^2 \sqrt{N}$ and  all of $\gcd(b, a+{b})$, $\gcd({b'}, a+{b})$, $\gcd(a+b', a+{b})$ at most $O(\delta^{-2})$ and satisfying
\begin{equation}\label{PHbb'}
\sum_x g_b(x) \overline{\tilde{h}_{b'}(x)g_{b'}(x)} \tilde{h}_{b}(x) \gg \delta^2 N.
\end{equation}

Expanding the definition of $\tilde{h}_{b'}$, using the invariance of $g_b$ and changing variables gives
\begin{multline*}
 \sum_x \E_{k_1, k_3 \in [K]}\sum_{k_2}\mu_K(k_2) g_b(x+ (a+b')k_2 + {b'}k_3) \overline{h_{b'}(x+bk_1 + {b'}k_3)}\\ \overline{g_{b'}(x+bk_1 + (a+b')k_2)}\ \tilde{h}_{b}(x+bk_1 + (a+b')k_2 + {b'}k_3) \gg \delta^2  N .
\end{multline*}
Since $h_{b'}$ is supported on $[N]$ and $|b|, |b'|, K \leq  N^{1/2}$, there are at most $O(N)$ values of $x$  which contribute to the above sum.  Applying H\"older's inequality then gives
\begin{multline*}
 \sum_x \biggbrac{\E_{k_1, k_3 \in [K]}\sum_{k_2}\mu_K(k_2) g_b(x+ (a+b')k_2 + {b'}k_3) \overline{h_{b'}(x+bk_1 + {b'}k_3)}\\ \overline{g_{b'}(x+bk_1 + (a+b')k_2)}\ \tilde{h}_{b}(x+bk_1 + (a+b')k_2 + {b'}k_3) }^8\gg \delta^{16}  N .
\end{multline*}
The sum inside the 8th power corresponds to an  integral with respect to three probability measures on $\Z$, with integrand amenable to Lemma \ref{box cauchy}.  Combining this with a change of variables gives
$$
 \sum_x\sum_{k_1, k_2, k_3}\mu_K(k_1)\nu_K(k_2)\mu_K(k_3) \Delta_{bk_1, (a+b')k_2, {b'}k_3}\ \tilde{h}_{b}(x) \gg \delta^{16}  N ,
$$
where we set
$$
\nu_K(k) : = \sum_{k_1 - k_2 = k} \mu_K(k_1)\mu_K(k_2).
$$

By Lemma \ref{h lipschitz} and the definition \eqref{eq:hb-def}, each $\tilde{h}_b$ is $O(K^{-1})$-Lipschitz along $(a+b)\cdot \Z$.  Hence, if $l_i \in [L]$, a telescoping identity shows that
$$
|\Delta_{h_1+(a+{b})l_1, h_2+(a+{b})l_2, h_3+(a+{b})l_3} \tilde{h}_{b}(x)- \Delta_{h_1, h_2, h_3} \tilde{h}_{b}(x)| \ll L/K.
$$
Taking $L := c \delta^{16} K$ we obtain 
\begin{multline*}
 \sum_x\sum_{k_1, k_2, k_3}\mu_K(k_1) \nu_K(k_2) \mu_K(k_3)\E_{l_1, l_2, l_3\in [L]}  \\ \Delta_{bk_1 + (a+{b})l_1,\, (a+b')k_2+ (a+{b})l_2,\, {b'}k_3+ (a+{b})l_3}\ \tilde{h}_{b}(x) \gg \delta^{16}  N .
\end{multline*}
We may replace the uniform measure on the $l_i$ by Fej\'er kernels at the cost of three applications of Cauchy--Schwarz; this gives
\begin{multline*}
 \sum_x\sum_{\substack{k_1, k_2, k_3\\l_1, l_2, l_3}} \mu_K(k_1) \nu_K(k_2) \mu_K(k_3) \mu_L(l_1) \mu_L(l_2) \mu_L(l_3)\\ \Delta_{bk_1 + (a+{b})l_1,\, (a+b')k_2+ (a+{b})l_2,\, {b'}k_3+ (a+{b})l_3}\ \tilde{h}_{b}(x) \gg \delta^{O(1)}  N .
\end{multline*}

Write 
\begin{align*}
\lambda_1(h)   := \sum_{bk + (a+{b})l = h}& \mu_K(k) \mu_L(l),\qquad
\lambda_2(h)  := \sum_{(a+b')k + (a+{b})l = h} \nu_K(k) \mu_L(l),\\
& \lambda_3(h)  := \sum_{{b'}k + (a+{b})l = h} \mu_K(k) \mu_L(l).
\end{align*}
Then
$$
 \sum_x\sum_{h_1, h_2, h_3} \lambda_1(h_1)\lambda_2(h_2)\lambda_3(h_3) \\ \Delta_{h_1, h_2, h_3}\ \tilde{h}_{b}(x) \gg \delta^{O(1)}  N .
$$

By Fourier inversion and extraction of a large Fourier coefficient, there exist $\alpha_i \in \T$ such that 
$$
\biggabs{ \sum_x\sum_{h_1, h_2, h_3}  \Delta_{h_1, h_2, h_3}\ \tilde{h}_{b}(x)e(\underline{\alpha} \cdot \underline{h})}\prod_{i=1}^3 \int_\T \bigabs{\widehat{\lambda}_i(\beta)}\intd\beta \gg \delta^{O(1)}  N .
$$
By our choice of $b$, $b'$ (see the paragraph preceding \eqref{PHbb'}), together with Lemma \ref{L1 Fourier bound}, for each $i$ we have 
\begin{equation}\label{fejer fourier bound}
\int_\T \bigabs{\widehat{\lambda}_i(\alpha)}\intd\alpha \ll \frac{\delta^{-8}}{KL} \ll \frac{\delta^{-O(1)}}{ N},
\end{equation}
the latter following from the fact that $L \gg c\delta^{16} K$ and $K \geq \delta N^{1/2}$.
On combining this with Lemma \ref{phase invariance} we obtain
$$
\bignorm{\tilde{h}_{b}}_{U^3}^8 \gg \delta^{O(1)} N^4.
$$
Since $\tilde{h}_b$ is an average of translates of $h_b$, we may apply the triangle inequality for the $U^3$-norm, together with the fact that Gowers norms are translation invariant, and conclude that $\norm{h_b}_{U^3}^8 \gg \delta^{O(1)} N^4$. Summing over $b \in \mathcal{B}$ gives our final bound.
\end{proof}

Finally we synthesise Lemmas \ref{linearisation}, \ref{arithcor} and \ref{hb lemma}.

\begin{theorem}[Global $U^5$-control]\label{global U5}
Let $g_0, g_1, f : \Z \to \C$ be 1-bounded functions, each supported in $[N]$.  Suppose that 
$$
\abs{\sum_{x \in \Z} \sum_{y \in \N} g_0(x)g_1(x+y) f(x+qy^2)} \geq \delta \sum_{x \in \Z} \sum_{y \in \N} 1_{[N]}(x)1_{[N]}(x+y) 1_{[N]}(x+qy^2).
$$
Then either $N \ll q$ or
$$
\sum_{u \in [q]}\norm{f}_{U^5(u + q  \Z)}^{2^5} \gg \delta^{O(1)} \sum_{u \in [q]}\norm{1_{[N]}}_{U^5(u+ q  \Z)}^{2^5}.
$$
\end{theorem}

\begin{proof} We recall our conventions \eqref{M def} and \eqref{counting op} regarding $M$ and $\Lambda_q$, and note that $\Lambda_q(1_{[N]}) \gg NM$ unless $N \ll q$.
We begin by applying the linearisation procedure (Lemma \ref{linearisation}) to deduce that
$$
 \sum_{ a,b\in (-2M,2M)}\ \biggabs{\sum_{h}\mu_{H}(h)\sum_x\Delta_{q(a+b)h_1,qbh_2,qah_3}f(x)}\\\gg \delta^{32} NM^2.
$$
Applying Lemma \ref{gcd} we may discard those $a,b$ for which either $\gcd(a,b) > C\delta^{-32}$ or $|b| < c\delta^{32} M$.
Partitioning the sum over $x$ into congruence classes $u \bmod q$, the popularity principle gives:
\begin{itemize}
\item at least $\Omega(\delta^{32} q)$ residues $u \in [q]$;
\item for each of which there is a subset of $h_3 \in (-H, H)$ of $\mu_H$-measure\footnote{i.e.\ $\sum_{h_3 \in \mathcal{H}} \mu_H(h_3) \gg \delta^{32}$.} at least $\Omega(\delta^{32})$; 
\item for each of which there exist $\Omega(\delta^{32} M)$ values of $a \in (-2M,2M)$;
\item for each of which there are $\Omega(\delta^{32}M)$ values of $b \in (-2M,2M)$ satisfying $\gcd(a,b) \ll \delta^{-32}$ and $|b| \gg \delta^{32} M$;
\end{itemize}
and together these satisfy
\begin{equation*}
 \biggabs{\sum_{h_1, h_2}\mu_{H}(h_1, h_2)\sum_x\Delta_{(a+b)h_1,bh_2,ah_3}f(qx-u)}\\\gg \delta^{32} M^2.
\end{equation*}

For fixed $u, h_3, a$ write 
$
\tilde{f}(x) := \Delta_{ah_3} f(qx-u),
$ so that $\tilde{f}$ has support in the interval $[(2M)^2]$ and
\begin{equation*}
 \biggabs{\sum_{h_1, h_2}\mu_{H}(h_1, h_2)\sum_x\Delta_{(a+b)h_1,bh_2}\tilde{f}(x)}\\\gg \delta^{32} M^2.
\end{equation*}

Set
\begin{equation}\label{1st H bound}
H:= c \delta^{96} M \qquad \text{and}\qquad K:= c^3 \delta^{256} M,
\end{equation}
with $c$ sufficiently small to ensure that we may apply Lemma \ref{arithcor}. This gives the existence of a 1-bounded $b$-periodic function $g_b$ such that on setting
\begin{equation}\label{hb defn}
\tilde{h}_b(x) := \sum_k\mu_K(k)  \overline{\tilde{f}(x+(a+b)k)g_b(x+(a+b)k)} 
\end{equation}
we have
\[
\sum_x\tilde{f}(x)g_b(x)\tilde{h}_b(x)\gg \delta^{448} M^2.
\]

Setting $\eta := c \delta^{480}$ for some small absolute constant $c >0$, we may sum over our set of permissible $b$ to deduce that
$$
\sum_{\substack{\eta M \leq |b| \leq 2M\\ \gcd(a, b) \leq \eta^{-1}}} \sum_x\tilde{f}(x)g_b(x)h_b(x)\geq \eta M^3.
$$
The hypotheses of Lemma \ref{hb lemma} having been met, we conclude that
$$
\E_{ |b| \leq 2M} \bignorm{\tilde{f}g_b}_{U^3}^8 \gg \delta^{O(1)}  \norm{1_{[M^2]}}_{U^3}^8.
$$
Applying Lemma \ref{difference vs periodic product} then gives
$$
\bignorm{\tilde{f}}_{U^4}^{16} \gg \delta^{O(1)}  \norm{1_{[M^2]}}_{U^4}^{16}.
$$

Recalling that $\tilde{f}(x) = \Delta_{ah_3} f_u(x)$ where $f_u(x) := f(qx-u)$, we may integrate over the set of permissible $h_3 $ and $a$, utilising positivity to extend the range of summation, and deduce that
$$
\E_{|a|\leq 2M} \sum_h \mu_H(h_3) \bignorm{\Delta_{ah_3} f_u}_{U^4}^{16} \gg \delta^{O(1)}  \norm{1_{[M^2]}}_{U^4}^{16}
$$
Using Lemma \ref{densifying difference functions} and summing over the permissible range of $u$ we get that
$$
 \E_{u\in [q]}\norm{f_u}_{U^5}^{32} \gg \delta^{O(1)}  \norm{1_{[M^2]}}_{U^5}^{32},
$$
and the result follows.
\end{proof}

\section{Degree lowering}\label{sec6}
So far, we have shown that $\Lambda_q(f_0,f_1,f_2)$ is controlled by $\E_{u\in[q]}\|f_2\|_{U^5(u+q\Z)}^{2^5}$ whenever $f_0,f_1,$ and $f_2$ are $1$-bounded complex-valued functions supported on the interval $[N]$. The next step in our argument is to bound $\Lambda_q(f_0,f_1,f_2)$ in terms of the $U^5(u+q\Z)$-norm of the dual function
\begin{equation}\label{eq6.1}
F(x):=\E_{y\in[M]}f_0(x-qy^2)f_1(x+y-qy^2).
\end{equation}
We postpone this deduction until \S\ref{inverse theorem section}. In this section we show how $U^5$-control of the dual implies $U^2$-control.  

Our argument combines three simple lemmas: Weyl's inequality; what we call `dual--difference interchange', which allows us to replace the difference function of the dual by the dual of the difference functions; and the fact that a function whose difference functions correlate with `low rank' Fourier coefficients must have a large uniformity norm of lower degree.

The following log-free variant of Weyl's inequality can be found in \cite[Lemma A.11]{GreenTaoQuadratic}.

\begin{lemma}[Weyl's inequality]\label{weyl-ineq}
Let $\alpha,\beta \in \T$, $\delta \in (0,1)$ and let $I \subset \Z$ be an interval. Suppose that
\[ \big| \E_{y \in I} e(\alpha y^2 + \beta y ) \big| \geq \delta.\]
Then either $|I|\ll \delta^{-O(1)}$ or there exists a positive integer $q \ll \delta^{-O(1)}$ such that  
$$\|q\alpha\| \ll \delta^{-O(1)}|I|^{-2}.$$
\end{lemma}

This has the following consequence, which does not necessarily assume our convention \eqref{M def} regarding $M$.

\begin{lemma}\label{lem6.2}
Suppose that for $\alpha \in \T$ there are $1$-bounded functions $g_0,g_1:\Z\to\C$ supported on the interval $[N]$ such that
\[
\left|\sum_x\sum_{y\in[M]}g_0(qx)g_1(qx+y)e(\alpha y^2)\right|\geq\delta M N/q.
\]
Then either $M \ll q\delta^{-O(1)}$ or there exists a positive integer $q'\ll\delta^{-O(1)}$ such that $\|q'q^2\alpha\|\ll \delta^{-O(1)}q^2/M^2$.
\end{lemma}
\begin{proof}
We split the sum over $y\in[M]$ into arithmetic progressions modulo $q$ and split the sum over $x $ into intervals of length $M/q$.  Hence, by the pigeon-hole principle, there exists $u \in [q]$ and an integer $m$ such that on rounding the sum over $y$ we have 
$$
\left|\sum_{x, y\in[M/q]}g_0(q(m+x))g_1(u+q(m+x+y))e\brac{\alpha (u+qy)^2}\right| \\ \gg\delta(M/q)^2.
$$
Define the functions
\begin{align*}
h_0(x) := g_0(q(m+x)) &  1_{[M/q]}(x),  \qquad h_1(x):= g_1(u+q(m+x))1_{[2M/q]},\\  &h_2(x) := e\brac{\alpha (u+qx)^2}1_{[M/q]}(x)
\end{align*}
Then by orthogonality, extraction of a large Fourier coefficient and Parseval we have
\begin{align*}
\delta M^2/q^2 \ll \left|\int_\T \hat{h}_0(\beta) \hat{h}_1(-\beta) \hat{h}_2(\beta) \intd\alpha\right| \ll \bignorm{\hat{h}_2}_\infty \bignorm{\hat{h}_0}_{L^2} \bignorm{\hat{h}_1}_{L^2} \ll \bignorm{\hat{h}_2}_\infty M/q.
\end{align*}

It follows that there exists $\beta \in \T$ such that
$$
\abs{\sum_{x \in [M/q]} e\brac{\alpha (u+qx)^2+ \beta x}} \gg\delta M/q.
$$
Applying Weyl's inequality, we deduce the existence of $q' \ll \delta^{-O(1)}$ such that $\norm{q' q^2 \alpha} \ll \delta^{-O(1)} /(M/q)^2 =\delta^{-O(1)} q^2/M $.
\end{proof}

\begin{lemma}[Dual--difference interchange]\label{dual difference interchange}
For each $y \in [M]$, let $F_y : \Z \to \C$ be a 1-bounded function with support in an interval of length $N$.  Set
$$
F(x) := \E_{y \in [M]} F_y(x).
$$
Then for any function $\phi: \Z^s \to \T$ and  finite set $\mathcal{H}\subset \Z^s$ we have
\begin{multline*}
\brac{N^{-s-1}\sum_{\vh \in \mathcal{H}}\abs{ \sum_x  \Delta_{\vh} F(x) e\bigbrac{\phi(\vh)x }}}^{2^s} \ll_s \\N^{-2s-1} \sum_{\vh^{0}, \vh^{1}\in \mathcal{H}}\abs{\sum_x \E_{y \in [M]} \Delta_{\vh^0-\vh^1} F_y(x) e\bigbrac{\phi(\vh^0; \vh^1)x }},
\end{multline*}
where
$$
\phi(\vh^0; \vh^1) := \sum_{\omega \in \set{0, 1}^s} (-1)^{|\omega|} \phi(\vh^{\omega})\qquad \text{and} \qquad
\vh^{\omega} := (h_1^{\omega_1}, \dots, h_s^{\omega_s}).
$$
\end{lemma}

\begin{proof}
We proceed by induction on $s \geq 0$, the base case being an identity.
Suppose then that $s \geq 1$.  For $\vh \in \Z^{s-1}$ and $h \in \Z$, we note that
\begin{equation}\label{eq:delta-F}
\Delta_{(\vh, h)} F(x) = \Delta_{\vh}\brac{ \E_{y, y' \in [M]} F_y(x)\overline{F_{y'}(x+h)}}.
\end{equation}
Furthermore, if $(\vh,h) \in \mathcal{H}$ contributes a non-zero expression of the form \eqref{eq:delta-F} then $(\vh, h) \in (-N, N)^s$, since the support of $F$ is contained in an interval of length $N$. 
Hence by the induction hypothesis
\begin{multline*}
 \brac{N^{-s-1}\sum_h\sum_{\substack{\vh \\ (\vh,h) \in \mathcal{H}}}\abs{ \sum_x  \Delta_{(\vh,h)} F(x) e\bigbrac{\phi(\vh)x }}}^{2^s} \ll_s  \\
 \brac{N^{-2s}\sum_h  \sum_{\substack{\vh^{0}, \vh^{1}\\ (\vh^i,h) \in \mathcal{H}}}\abs{\sum_x \E_{y ,y'\in [M]} \Delta_{\vh^0-\vh^1} F_y(x)\overline{F_{y'}(x+h)} e\bigbrac{\phi(\vh^0;\vh^1;h)x }}}^2 ,
\end{multline*}
where
$$
\phi(\vh^0; \vh^1;h) := \sum_{\omega \in \set{0, 1}^{s-1}} (-1)^{|\omega|} \phi(\vh^{\omega},h).
$$
Letting $e(\psi(\vh^0;\vh^1;h))$ denote the conjugate phase of the inner absolute value, we take the sum over $h$ inside and apply Cauchy--Schwarz to obtain
\begin{align*}
&\brac{\sum_{\vh^{0}, \vh^{1},x}  \E_{y ,y'\in [M]} \sum_{\substack{h\\ (\vh^i,h) \in \mathcal{H}}}\Delta_{\vh^0-\vh^1} F_y(x)\overline{F_{y'}(x+h)} e\bigbrac{\phi(\vh^0;\vh^1;h)x +\psi(\vh^0;\vh^1;h)}}^2
\\ 
&\ll_s N^{2s-1} \sum_{\vh^{0}, \vh^{1}}  \sum_{\substack{h^0, h^1\\ (\vh^i,h^j) \in \mathcal{H}}}\\ 
& \qquad \abs{\sum_x \E_{y \in [M]} \Delta_{\vh^0-\vh^1} F_y(x)\overline{F_{y}(x+h^0-h^1)} e\Bigbrac{\bigbrac{\phi(\vh^0;\vh^1;h^0)-\phi(\vh^0;\vh^1;h^1)}x}}.
\end{align*}
The result follows.
\end{proof}

If $\phi(h_1, \dots, h_{s-1})$ is a function of $s-1$ variables we write
$
\phi(h_1, \dots, \hat{h}_i, \dots, h_{s}) := \phi(h_1, \dots, h_{i-1}, h_{i+1}, \dots, h_{s}).
$
We say that $\phi(h_1, \dots, h_s)$ is \emph{low rank} if there exist functions $\phi_i(h_1, \dots, h_{s-1})$ such that
$$
\phi(h_1, \dots, h_s) = \sum_{i=1}^s \phi_i(h_1, \dots, \hat{h}_i, \dots, h_s).
$$
From the definition of the Gowers norm together with the $U^2$-inverse theorem (Lemma \ref{U2 inverse}), one can show that largeness of the $U^{s+2}$-norm is equivalent to the existence of $\phi :\Z^s \to \T$ such that
$$
\sum_{h_1, \dots, h_s} \abs{\sum_x \Delta_h f(x)e(\phi(h)x)} \gg N^{s+1}.
$$
The following lemma says that if $\phi$ is low-rank, then the $U^{s+1}$-norm must also be large.
\begin{lemma}[Low rank correlation implies lower degree]\label{low rank lemma}
Let $f : \Z \to \C$ be a 1-bounded function with support in $[N]$.  Then for $\phi_1, \dots, \phi_m : \Z^{s-1} \to \T$ with $m \leq s$ we have
\begin{equation}\label{low rank ineq}
 \recip{N^{s+1}} \sum_{h_1, \dots, h_s}\abs{\sum_x \Delta_{h} f(x) e\brac{\sum_{i=1}^m\phi_i(h_1, \dots, \hat{h}_i, \dots, h_s)x}}\\ \ll_s  \brac{\frac{\norm{f}_{U^{s+1}}^{2^{s+1}}}{N^{s+2}}}^{2^{-m-1}}.
\end{equation}
 \end{lemma}

\begin{proof}
We proceed by induction on $m \geq 0$, the base case corresponding to the Cauchy--Schwarz inequality.  Suppose then that $m \geq 1$ and the result is true for smaller values of $m$.  Letting $e(\psi(h))$ denote the conjugate phase of the inner-most sum, the left-hand side of \eqref{low rank ineq} is equal to
\begin{multline*}
 \recip{N^{s+1}} \sum_{h_2, \dots, h_s, x}\Delta_{h_2, \dots, h_s}f(x)e\brac{\phi_1(h_2, \dots, h_s)}\sum_{h_1} \Delta_{h_2, \dots, h_s} \overline{f(x+h_1)} \\
 e\brac{\sum_{i=2}^m\phi_i(h_1, \dots, \hat{h}_i, \dots, h_s)x+ \psi(h_1, \dots, h_s)}.
\end{multline*}
By Cauchy--Schwarz, the square of this is at most
\begin{multline*}
 \ll_s \recip{N^{s+2}} \sum_{h_2, \dots, h_s}\ \sum_{h_1, h_1'\in (-N, N)}\\\abs{\sum_x\Delta_{h_1-h_1',h_2, \dots, h_s} f(x) 
 e\brac{\sum_{i=2}^m\brac{\phi_i(h_1, \dots, \hat{h}_i, \dots, h_s)-\phi_i(h_1', \dots, \hat{h}_i, \dots, h_s)}x}}.
\end{multline*}
Taking a maximum over $h_1' \in (-N, N)$ and changing variables in $h_1$, the latter is at most an absolute constant times
\begin{multline*}
 \recip{N^{s+1}} \sum_{h_1, h_2, \dots, h_s}\Bigg|\sum_x\Delta_{h_1,h_2, \dots, h_s} f(x) \\
 e\brac{\sum_{i=2}^m\brac{\phi_i(h_1+h_1',h_2 \dots, \hat{h}_i, \dots, h_s)-\phi_i(h_1',h_2 \dots, \hat{h}_i, \dots, h_s)}x}\Bigg|.
\end{multline*}
This phase  has lower rank than the original, hence we may apply the induction hypothesis to yield the lemma.
\end{proof}

\begin{lemma}[Degree lowering]\label{degree lowering lemma}
Let $f_0, f_1 : \Z \to \C$ be 1-bounded functions with support in $[N]$ and, writing $M := \sqrt{N/q}$, define the dual 
$$
F(x) := \E_{y \in [M]} f_0(x-qy^2) f_1(x+y-qy^2).
$$
If, for $s \geq 3$, we have
$$
\sum_{u \in [q]}\norm{F}_{U^s(u + q \cdot \Z)}^{2^s} \geq \delta \sum_{u \in [q]}\norm{1_{[N]}}_{U^s(u + q \cdot \Z)}^{2^s},
$$
then either $N \ll_s q^3\delta^{-O_s(1)}$ or
$$
\sum_{u \in [q]}\norm{F}_{U^{s-1}(u + q \cdot \Z)}^{2^{s-1}} \gg_s \delta^{O_s(1)} \sum_{u \in [q]}\norm{1_{[N]}}_{U^{s-1}(u + q \cdot \Z)}^{2^{s-1}},
$$
\end{lemma}

\begin{proof}  Write $M := \floor{(N/q)^{1/2}}$.  
Given $u \in [q]$ let 
$
F_u(x) := F(u+qx)$, a function with support in the interval $[2N/q]$.  Applying the popularity principle, there exists a set of $\Omega_s(\delta q)$ residues $u \in [q]$ for which $\norm{F_u}_{U^s}^{2^s} \gg_s \delta (N/q)^{s+1}$.  Expanding the definition of the $U^s$-norm \eqref{Us def} we have
$$
\sum_{h_1, \dots, h_{s-2}} \norm{\Delta_{h_1, \dots, h_{s-2}}F_u}_{U^2}^4 \gg_s \delta (N/q)^{s+1}.
$$
Applying the $U^2$-inverse theorem (Lemma \ref{U2 inverse}), there exists $\mathcal{H} \subset (-2N/q, 2N/q)^{s-2}$ of size $|\mathcal{H}| \gg_s \delta (N/q)^{s-2}$ and a function $\phi : \Z^{s-2} \to \T$ such that for every $\vh \in \mathcal{H}$ we have
\begin{equation}\label{phi correlation}
\abs{\sum_x \Delta_{\vh} F_u(x) e\bigbrac{\phi(\vh)x} }\gg_s \delta N/q.
\end{equation}

Set $T := \ceil{C_s\delta^{-1}N/q}$, with $C_s$ an absolute constant taken sufficiently large to ensure that, on rounding $\phi(\vh)$ to the nearest fraction of the form $t/T$, the validity of \eqref{phi correlation} remains.  Summing over $\vh \in \mathcal{H}$ and applying Lemma \ref{dual difference interchange}, we deduce that 
\begin{multline*}
 \sum_{\vh^{0}, \vh^{1}\in \mathcal{H}} \biggl|\sum_x\E_{y \in [M]} \Delta_{\vh^0-\vh^1} f_0(u+qx-qy^2)\Delta_{\vh^0-\vh^1}f_1(u+qx+y-qy^2)\\
  e\bigbrac{\phi(\vh^0; \vh^1)x}\biggr| \gg_s \delta^{O_s(1)} (N/q)^{2s-3}.
\end{multline*}
Applying the pigeon-hole and popularity principle, there exists $\mathcal{H}' \subset \mathcal{H}$ of size $\gg_s \delta^{O_s(1)} (N/q)^{s-2}$ and $\vh^{1} \in \mathcal{H}$ such that for every $\vh^0 \in \mathcal{H}'$ we have
\begin{multline*}
 \abs{\sum_x \sum_{y \in [M]} \Delta_{\vh^0-\vh^1} f_0(u+qx-qy^2)\Delta_{\vh^0-\vh^1}f_1(u+qx+y-qy^2)
  e\bigbrac{\phi(\vh^0, \vh^1)x }}\\ \gg_s \delta^{O_s(1)} MN/q.
\end{multline*}

By Lemma \ref{lem6.2}, for each $\vh^0\in \mathcal{H}'$ there exists $q' \ll \delta^{-O_s(1)}$ such that 
$$
\norm{q'q^2\phi(\vh^0, \vh^1)}\ll \delta^{-O_s(1)}q^3/N.
$$  
Notice that $\phi(\vh^0, \vh^1)$ is an element of the additive group $\set{t/T : t \in [T] }\subset \T$.  Moreover, for any $Q$ we have the inclusion 
$$
\set{\alpha \in \T : \exists q' \leq Q \text{ with } \norm{q'q^2\alpha}\leq Q q^3/N} \subset 
\bigcup_{1 \leq a \leq q' \leq Q}  \sqbrac{ \frac{a}{q'q^2} - \frac{qQ}{N}, \frac{a}{q'q^2} + \frac{qQ}{N}}.
$$
By a volume packing argument, the number of $t/T$ lying in this union of intervals is at most $Q^2(1+\tfrac{2qQT}{N})\ll_s \delta^{-O_s(1)}$.  It therefore follows from the pigeon-hole principle that there exists $\mathcal{H}''\subset \mathcal{H}'$ of size $\gg_s\delta^{O_s(1)} (N/q)^{s-2}$ and $t_0 \in [T]$ such that for any $\vh^0\in \mathcal{H}''$ we have $\phi(\vh^0, \vh^1) = t_0/T$.  In particular, when restricted to the set $\mathcal{H}''$, the function $\phi$ satisfies
$$
\phi(\vh^0) =  t_0/T - \sum_{\omega\in \set{0,1}^{s-2}\setminus\set{0}} (-1)^{|\omega|} \phi(\vh^{\omega}).
$$
The right-hand side of this identity is  \emph{low rank} according to the terminology preceding Lemma \ref{low rank lemma}.

Summing over $\vh\in \mathcal{H}''$ in \eqref{phi correlation}, we deduce the existence of a  low rank function $\psi: \Z^{s-2} \to \T$ such that
$$
\sum_{\vh}\abs{\sum_x \Delta_h F_u(x) e\bigbrac{\psi(\vh)x} }\gg_s \delta^{O_s(1)} (N/q)^{s-1}.
$$
Employing Lemma \ref{low rank lemma} then gives
$$
\norm{F_u}_{U^{s-1}}^{2^{s-1}} \gg_s \delta^{O_s(1)}(N/q)^s.
$$
Summing over permissible $u$, then extending to the full sum over $ u\in [q]$ by positivity, we obtain the bound claimed in the lemma.
\end{proof}

\section{The inverse theorem for our counting operator}\label{inverse theorem section}

In this section we show how $U^5$-control of the final function in our counting operator, as proved in Theorem \ref{global U5}, also yields $U^5$-control of the dual function. Combining this with the degree lowering of \S\ref{sec6}, we deduce that the dual is controlled by the $U^1$-norm. This allows us to deduce the following key inverse theorem for our counting operator.

\begin{theorem}[Inverse theorem for nonlinear Roth]\label{thm:inverse-one}
Let $f_0, f_1, f_2 : \Z \to \C$ be 1-bounded functions, each with support in $[N]$.  Suppose that
$$
\abs{\sum_{x \in \Z} \sum_{y \in \N} f_0(x)f_1(x+y) f_2(x+qy^2)} \geq \delta \sum_{x \in \Z} \sum_{y \in \N} 1_{[N]}(x)1_{[N]}(x+y) 1_{[N]}(x+qy^2).
$$
Then either $N \ll q^3 \delta^{-O(1)}$, or there exists $q'\ll \delta^{-O(1)}$ such that for each $i=0,1,2$ there exists a 1-bounded function $\phi_i : \Z \to \C$ which is $\ll \delta^{-O(1)} q^{3/2}N^{-1/2}$ Lipschitz along $q'q\cdot \Z$, in that
\begin{equation}\label{eq:inverse-lip}
|\phi_i(x+q'qy)-\phi_i(x)| \ll \delta^{-O(1)}q^{3/2}N^{-1/2} |y|\qquad (\forall x,y \in \Z),
\end{equation}
and for which
$$
\abs{\sum_x f_i(x)\phi_i(x)} \gg \delta^{O(1)} N.
$$
\end{theorem}

\begin{remark}
Inspection of the following proof reveals that for $i = 1$ one can in fact ensure that  the function $\phi_1$ is $\ll \delta^{-O(1)} q^{1/2}N^{-1/2}$ Lipschitz along $q'\cdot \Z$. This stronger property is not needed in our present application.
\end{remark}

\begin{proof}
Define the dual 
\begin{equation}\label{dual defn}
F_2(x) := \E_{y \in [M]} f_0(x-qy^2) f_1(x+y-qy^2).
\end{equation}
Either $N \ll q$ or we have
$$
\delta NM \ll |\Lambda_q(f_0, f_1, f_2)| = M\abs{ \sum_x F_2(x) f_2(x)}.
$$
Since $f_2$ is supported on $[N]$, the Cauchy-Schwarz inequality gives
$$
\Lambda_q(f_0, f_1, \overline{F_2}) = M\sum_x F_2(x)\overline{F_2(x)} \gg \delta^2NM.
$$
Since the functions $f_0, f_1, \overline{F_2}$ all have support contained in $[2N]$, we may apply Theorem \ref{global U5} to deduce that
$$
\sum_{u \in [q]}\norm{F_2}_{U^5(u + q \cdot \Z)}^{2^5} \gg \delta^{O(1)} \sum_{u \in [q]}\norm{1_{[2N]}}_{U^5(u + q \cdot \Z)}^{2^5}.
$$
We now apply Lemma \ref{degree lowering lemma} three times to obtain
$$
\sum_{u \in [q]}\norm{F_2}_{U^2(u + q \cdot \Z)}^{4} \gg \delta^{O(1)} \sum_{u \in [q]}\norm{1_{[2N]}}_{U^2(u + q \cdot \Z)}^{4}.
$$

By the popularity principle, there are at least $\gg \delta^{O(1)} q$ values of $u \in [q]$ for which $\norm{F_2}_{U^2(u + q \cdot \Z)}^{4} \gg \delta^{O(1)} \norm{1_{[2N]}}_{U^2(u + q \cdot \Z)}^{4}$.  The inverse theorem for the $U^2$-norm then gives the existence of $\phi(u) \in \T$ for which 
\begin{equation}\label{u correlation}
\abs{\sum_x F_2(u + qx) e(\phi(u) x)} \gg \delta^{O(1)} N/q.
\end{equation}
Set $T := \ceil{C\delta^{-C}N/q}$, with $C$ an absolute constant taken sufficiently large to ensure that, on rounding $\phi(u)$ to the nearest fraction of the form $t/T$, the inequality \eqref{u correlation} remains valid.  

By Lemma \ref{lem6.2}, for each $u$ satisfying \eqref{u correlation}, there exists a positive integer $q'\ll\delta^{-O(1)}$ such that $\|q'q^2\phi(u)\|\ll\delta^{-O(1)}q^3/N$.  By a volume packing argument similar to that given in the proof of Lemma \ref{degree lowering lemma}, the function $\phi$ is constant on a proportion of at least $\gg \delta^{O(1)}$ of the residues $u \in [q]$ satisfying \eqref{u correlation}.  Summing over these $u$, then extending the sum to all of $[q]$, we deduce the existence of $\alpha \in \T$  such that 
\begin{equation}\label{Gu discrep}
\sum_{u \in [q]} \abs{\sum_x F_2(u + qx) e(\alpha x)} \gg \delta^{O(1)} N.
\end{equation}

Expanding the dual function, there is a 1-bounded function $\psi(u\bmod q)$ such that the left-hand side of the above is equal to
\begin{multline}\label{hard correlation}
 \sum_{u \in [q]} \psi(u\bmod q)\sum_{x\equiv u(q)}  \E_{y \in [M]}f_0(x-qy^2)f_1(x+y-qy^2) e(\alpha x/q)\\
 = \sum_x f_0(x)\psi(x\bmod q)e(\alpha x/q) \E_{y \in [M]}f_1(x+y)e(\alpha y^2).
\end{multline}
Applying Lemma \ref{lem6.2} with $q = 1$, there exists $q' \ll \delta^{-O(1)}$ such that $\norm{q'\alpha} \ll  \delta^{-O(1)}/M^2 = \delta^{-O(1)} q/N$.

Using this, let us demonstrate the correlation of $f_0$ with a suitably Lipschitz function; the case of $f_1$ is similar (in fact simpler) and the case of $f_2$ is dealt with shortly.  Setting 
$$
\phi_0(x) := \psi(x\bmod q)e(\alpha x/q) \E_{y \in [M]}f_1(x+y)e(\alpha y^2),
$$
we have  $\sum_x f_0(x)\phi_0(x) \gg \delta^{O(1)} N$. Our aim is to show that $\phi_0$ is $\ll \delta^{-O(1)} q^{3/2}N^{-1/2}$ Lipschitz along $q'q\cdot \Z$.

For any $x, z \in \Z$ we have
\begin{multline*}
\abs{\phi_0(x) - \phi_0(x+q'q z)} \leq |1-e(\alpha q' z)|\\ + \abs{\E_{y\in [M]} f_1(x+y)e(\alpha y^2) - \E_{y\in [M]} f_1(x+y+q'qz)e(\alpha y^2)}.
\end{multline*}
The first term satisfies
$$
|1-e(\alpha q' z)| \ll \norm{q'\alpha} |z| \ll \delta^{-O(1)} q |z|/N\ll \delta^{-O(1)} q^{3/2}N^{-1/2} |z|.
$$
Changing variables in $y$, the second term is at most
\begin{align*}
 q'q|z|M^{-1}+&\abs{\E_{y\in [M]} g(x+y)e(\alpha y^2) - \E_{y\in [M]} g(x+y)e(\alpha (y-q'qz)^2)} \\
&\ll q'q|z|M^{-1}+ \E_{y \in [M]} \abs{e(\alpha y^2)-e(\alpha (y-q'qz)^2)}\\
& \ll q'q|z|M^{-1}+ \norm{\alpha(q'qz)^2}+\E_{y \in [M]}\norm{\alpha2yq'qz} 
\\
&\ll \delta^{-O(1)}q^{3/2}|z|N^{-1/2}+  \delta^{-O(1)} q^3 |z|^2N^{-1}.
\end{align*}
We thereby obtain the required Lipschitz inequality \eqref{eq:inverse-lip} in the case that $|z| \leq N^{1/2}q^{-3/2}$.  In the remaining case the Lipschitz inequality is trivial, since $\phi_0$ is 1-bounded.

Having proved the inverse theorem for $f_0$ and $f_1$, we now focus on $f_2$. As before, define the dual 
\begin{equation*}
F_1(x) := 1_{[N]}(x)\E_{y \in [M]} f_0(x-y) f_2(x+qy^2-y),
\end{equation*}
so that
$$
 \abs{ \sum_x F_1(x) f_1(x)} \gg \delta N.
$$
 Applying Cauchy-Schwarz  gives
$$
\Lambda_q(f_0, \overline{F}_1, f_2) = M\sum_x F_2(x)\overline{F_2(x)} \gg \delta^2NM,
$$
and the inverse theorem then yields $q_1\ll \delta^{-O(1)}$ such that there exists a 1-bounded function $\phi_1 : \Z \to \C$ which is $\ll \delta^{-O(1)} q^{3/2}N^{-1/2}$ Lipschitz along $q_1q\cdot \Z$ and for which
$$
\abs{\Lambda_q(f_0,\phi_1, f_2)} = \abs{\Lambda_q(f_0,\phi_11_{[N]}, f_2)}= M\abs{\sum_x F_1(x)\phi_1(x)} \gg \delta^{O(1)}NM.
$$

Repeating the above procedure, we obtain $q_0\ll \delta^{-O(1)}$ such that there exists a 1-bounded function $\phi_0 : \Z \to \C$ which is $\ll \delta^{-O(1)} q^{3/2}N^{-1/2}$ Lipschitz along $q_0q\cdot \Z$ and for which
$$
\abs{\Lambda_q(\phi_01_{[N]},\phi_1, f_2)}  \gg \delta^{O(1)}NM.
$$
We may replace $1_{[N]}$ in the above by a continuous function $\psi$ which is 1 on $[c\delta^C N, (1-c\delta^C)N]$, 0 on $\Z\setminus [N]$ and linear everywhere else. Since the linear parts have gradient $\ll \delta^{-O(1)}/N$, we deduce that $\psi$ is $\ll \delta^{-O(1)}/N$ Lipschitz along $\Z$. As a consequence, the product $\phi_0\psi$ is $\ll \delta^{-O(1)} q^{3/2}/N^{1/2}$  Lipschitz along $q_0 q \cdot \Z$. Finally, we observe that
$$
\delta^{O(1)}N \ll \max_{y\in [M]} \abs{\sum_x (\phi_0\psi)(x-qy^2) \phi_1(x+y-qy^2) f_2(x)},
$$
and the the function $x \mapsto (\phi_0\psi)(x-qy^2) \phi_1(x+y-qy^2)$ is $\ll \delta^{-O(1)} q^{3/2}/N^{1/2}$  Lipschitz along $q_0q_1 q \cdot \Z$.
\end{proof}

\begin{corollary}[Local correlation with constant functions]\label{cor:inverse}
Let $f_0, f_1, f_2 : \Z \to \C$ be 1-bounded functions, each with support in $[N]$.  Suppose that
$$
\abs{\sum_{x \in \Z} \sum_{y \in \N} f_0(x)f_1(x+y) f_2(x+qy^2)} \geq \delta \sum_{x \in \Z} \sum_{y \in \N} 1_{[N]}(x)1_{[N]}(x+y) 1_{[N]}(x+qy^2).
$$
Then either $N \ll q^3 \delta^{-O(1)}$, or there exists $q'\ll \delta^{-O(1)}$ and $N' \gg \delta^{O(1)} q^{-3/2} N^{1/2}$ such that for each $i=0,1,2$ we have 
$$
\sum_{x\in \Z}\left|\sum_{y\in[N']}f_i(x+q'qy)\right|\gg\delta^{O(1)}NN'.
$$
\end{corollary}

\begin{proof}
Let $\phi_i$ denote the function guaranteed by Theorem \ref{thm:inverse-one}. By the Lipschitz property of $\phi_i$ we have
$$
\abs{\phi_i(x) - \E_{y\in [N']} \phi_i(x+q'qy)} \ll \delta^{-O(1)} q^{3/2} N' N^{-1/2}.
$$
Hence taking $N' = c\delta^C N^{1/2} q^{-3/2}$ for $c>0$ sufficiently small and $C$ sufficiently large, we deduce that
$$
\abs{\sum_x f_i(x) \E_{y\in [N']} \phi_i(x+q'qy)} \gg \delta^{O(1)} N.
$$
Changing variables in $x$ and applying the triangle inequality yields the result.

\end{proof}

\section{The density increment lemma}\label{increment proof sec}
In this section we prove Lemma \ref{increment lemma}.
\begin{proof}[Proof of Lemma \ref{increment lemma}]
We either have $N \ll q$ or
$$
\sum_{x\in \Z}\sum_{y \in \N} 1_{[N]}(x)1_{[N]}(x+y)1_{A}(x+qy^2) \gg N^{3/2}q^{-1/2}.
$$
Therefore
\begin{multline*}
\abs{\sum_{x\in \Z}\sum_{y \in \N} 1_{A}(x)1_{A}(x+y)1_{A}(x+qy^2)-\delta^3 1_{[N]}(x)1_{[N]}(x+y)1_{[N]}(x+qy^2)}\\
 \gg \delta^3 N^{3/2}q^{-1/2}.
\end{multline*}

By a telescoping identity, there exist 1-bounded functions $f_0, f_1, f_2 : \Z \to \R$ all with support in $[N]$ and at least one of which is equal to $1_A - \delta 1_{[N]}$ such that
$$
\abs{\sum_{x\in \Z}\sum_{y \in \N} f_0(x)f_1(x+y)f_2(x+qy^2)}
 \gg \delta^3 N^{3/2}q^{-1/2}.
 $$
 Applying our inverse theorem (Corollary \ref{cor:inverse}) we deduce that there exists $q'\ll \delta^{-O(1)}$ and $N' \gg \delta^{O(1)} q^{-3/2} N^{1/2}$ such that 
$$
\sum_{x\in \Z}\left|\sum_{y\in[N']}\brac{1_A-\delta 1_{[N]}}(x+q'qy)\right|\gg\delta^{O(1)}NN'.
 $$
 Since the corresponding sum without absolute values is equal to zero, we are able to find $x$ such that
 $$
 \sum_{y\in[N']}\brac{1_A-\delta 1_{[N]}}(x+q'qy) \gg \delta^{O(1)} N'.
 $$
Writing $x_1 + q'q\cdot[N_1]$ for  $[N]\cap(x+q'q\cdot [N'])$ we have $N_1 \gg \delta^{O(1)} N' \gg \delta^{O(1)} q^{-3/2} N^{1/2}$ and
$$
|A \cap (x_1 + q'q\cdot [N_1])| \geq \brac{\delta + \Omega\brac{\delta^{O(1)}}}N_1.
$$
\end{proof}

\appendix

\section{Basic theory of the Gowers norms}

\begin{lemma}[Inverse theorem for the $U^2$-norm]\label{U2 inverse}
Let $f:\Z\to\C$ be a $1$-bounded function with support in $[N]$. 
Then there exists $\alpha \in \T$ such that
\[
\|f\|_{U^2}^4 \leq N\left|\sum_{x}f(x)e(\alpha x)\right|^2.
\]
\end{lemma}
\begin{proof}
Using the definition of the Fourier transform \eqref{Fourier transform}, together with orthogonality of additive characters, we have
$$
\norm{f}_{U^2}^4 = \int_{\T} \bigabs{\hat{f}(\alpha)}^4\intd\alpha \leq \bignorm{\hat{f}}_\infty^2 \int_{\T} \bigabs{\hat{f}(\alpha)}^2\intd\alpha \leq \bignorm{\hat{f}}_\infty^2 N.
$$
\end{proof}
For each $\omega\in\{0,1\}^s$, let $f_\omega:\Z\to\C$ be a function with finite support. Then we define the \emph{Gowers inner product} by
$$
[f_\omega]_{U^s} := \sum_{x, h_1, \dots, h_s} \prod_{\omega \in \set{0, 1}^s} \mathcal{C}^{|\omega|}f_\omega(x + \omega \cdot h).
$$
Here $\mathcal{C}$ denotes the operation of complex conjugation. Notice that $[f]_{U^s} = \norm{f}_{U^s}^{2^s}$.
\begin{lemma}[Gowers--Cauchy--Schwarz]
For each $\omega\in\{0,1\}^s$, let $f_\omega:\Z\to\C$ be a function with finite support. Then we have
\[
[f_\omega]_{U^s}\leq \prod_{\omega\in\{0,1\}^s}\|f_\omega\|_{U^s}.
\]
\end{lemma}

\begin{proof}
See \cite[Exercise 1.3.19]{TaoHigher}.
\end{proof}

\begin{lemma}[Phase invariance for $s \geq 2$]\label{phase invariance}
Let $L \in \R[x, h_1, \dots, h_s]$ be a linear form, with $s\geq 2$ and let $f : \Z \to \C$.  Then 
$$
\biggabs{\sum_{x, h_1, \dots, h_s} \Delta_{h_1, \dots, h_s} f(x) e(L(x, h_1, \dots, h_s))} \leq \norm{f}_{U^s}^{2^s}.
$$
\end{lemma}

\begin{proof}
The linear form may be written as
$$
L(x, h_1, \dots, h_s) = \alpha x + \beta_1(x+h_1) + \dots + \beta_s(x+h_s),
$$
for some real $\alpha$ and $\beta_i$.  Write $f_0(x) := f(x)e(\alpha x)$, $f_{e_i}(x) := f(x) e(-\beta_i x)$ for $i = 1, \dots, s$, and for $\omega \in \set{0,1}^s\setminus \set{0, e_1, \dots, e_s}$ set $f_\omega := f$. Then by Gowers--Cauchy--Schwarz we have
$$
\biggabs{\sum_{x, h_1, \dots, h_s} \Delta_{h_1, \dots, h_s} f(x) e(L(x, h_1, \dots, h_s))} \leq \prod_{\omega} \norm{f_\omega}_{U^s}.
$$
It therefore suffice to prove that for a phase function $e_\alpha : x \mapsto e(\alpha x)$ we have
$
\norm{fe_\alpha}_{U^s} = \norm{f}_{U^s}.
$
The latter follows on observing that 
$$
\Delta_{h_1, \dots, h_s} (fe_\alpha) = \brac{\Delta_{h_1, \dots, h_s}f}\brac{\Delta_{h_1, \dots, h_s}e_\alpha},
$$
and for any $x, h_1, \dots, h_s$ with $s \geq 2$ we have
$
\Delta_{h_1, \dots, h_s}e_\alpha(x) = 1.
$
\end{proof}

\begin{lemma}[Box Cauchy--Schwarz]\label{box cauchy}
Let $\mu_1, \mu_2, \mu_3$ be probability measures on $\Z$ with the discrete sigma algebra.  If $F_1, F_2, F_3$ are 1-bounded function on $\Z^2$ and $F$ is a 1-bounded function on $\Z^3$ then
\begin{multline*}
\abs{\sum_{x \in \Z^3} F_1(x_2, x_3)F_2(x_1, x_3) F_3(x_1, x_2) F(x_1,x_2,x_3)\mu(x_1)\mu(x_2)\mu(x_3)}^8\\ \leq \sum_{x^0, x^1\in \Z^3} \prod_{\omega \in \set{0,1}^3} \mathcal{C}^{|\omega|} F(x_1^{\omega_1}, x_2^{\omega_2},x_3^{\omega_3})\mu_1(x_1^{0})\mu_1(x_1^{1})\mu_2(x_2^{0})\mu_2(x_2^{1})\mu_3(x_3^{0})\mu_3(x_3^{1}).\end{multline*}

\end{lemma}


\end{document}